{tree2.eps}

\documentclass{amsart}
\usepackage{amsthm,amsmath,amssymb,amscd,graphics,epsfig,epic,eepic}

\newcommand{\entrylabel}[1]{\mbox{\textsf{{\rm c}1}}\hfil}
{\end{list}}
{
   \newtheorem{theorem}{Theorem}[subsection]
   
   \newtheorem{lemma}[theorem]{Lemma}

   \newtheorem{corollary}[theorem]{Corollary}

}
{\theoremstyle{definition}
   
   \newtheorem{example}[theorem]{Example}
   \newtheorem{definition}[theorem]{Definition}
}
{\theoremstyle{remark}
   \newtheorem*{remark}{Remark}
   \newtheorem*{remarks}{Remarks}
}
\newcommand{\Hi}{\cite{Hir}}

\newcommand{\Hii}{\cite{Hir1}}
\newcommand{\Vi}{\cite{Vi}}

\newcommand{\BM}{\cite{BM}}

\newcommand{\cCC}{{\mathbb{C}}}

\newcommand{\CC}{{\mathbb{C}}}

\newcommand{\ZZ}{{\mathbb{Z}}}

\newcommand{\bA}{{\mathbf{A}}}

\newcommand{\cC}{{\mathcal C}}
\newcommand{\cD}{{\mathcal D}}
\newcommand{\cE}{{\mathcal E}}
\newcommand{\cF}{{\mathcal F}}

\newcommand{\cH}{{\mathcal H}}
\newcommand{\cI}{{\mathcal I}}
\newcommand{\cJ}{{\mathcal J}}

\newcommand{\cM}{{\mathcal M}}
\newcommand{\cN}{{\mathcal N}}
\newcommand{\cO}{{\mathcal O}}

\newcommand{\cU}{{\mathcal U}}

\newcommand{\cT}{{\mathcal T}}

\newcommand{\supp}{{\operatorname{supp}}}

\newcommand{\res}{{\operatorname{res}}}

\newcommand{\Sub}{\operatorname{Sub}}

\newcommand{\Spec}{\operatorname{Spec}}

\newcommand{\Der}{{\operatorname{Der}}}

\newcommand{\ord}{{\operatorname{ord}}}

\newcommand{\codim}{\operatorname{codim}}

\newcommand{\Reg}{{\operatorname{Reg}}}

\begin{document}


\title[Effective Hironaka resolution and its Complexity ]{ Effective Hironaka resolution and its Complexity \\(with appendix on applications in positive characteristic)}

\author{{Edward Bierstone},~{Dima Grigoriev},~ {Pierre Milman},~ {Jaros{\l}aw W{\l}odarczyk}}
\thanks{Bierstone's and Milman's research was partially supported by NSERC Discovery Grants OGP 0009070 and OGP 008949}
\thanks{W\l odarczyk was supported in part by NSF grants DMS-0100598 and DMS-1201502}


\address{E. Bierstone, Fields Institute, 222 College Street, Toronto, Ontario, Canada M5T 3J1
and Department of Mathematics, University of Toronto, Toronto, Ontario, Canada M5S 2E4}
\email{bierston@fields.utoronto.ca}

\address{D. Grigoriev, Laboratoire des Math\'ematiques , Universit\'e de Lille I,
59655, Villeneuve d'Ascq, France\newline http://en.wikipedia.org/wiki/Dima\_Grigoriev}

\email{dmitry.grigoryev@math.univ-lille1.fr}

\address{P.D. Milman, Department of Mathematics, University of Toronto, Toronto, Ontario, Canada M5S 2E4}
\email{milman@math.toronto.edu}

\address{J. Wlodarczyk, Department of Mathematics\\Purdue University\\West
Lafayette, IN-47907\\USA}
\email{wlodar@math.purdue.edu}

\date{\today}
\begin{abstract}
Building upon works of Hironaka,
 Bierstone-Milman, Villamayor and W\l odarczyk, we give an
\emph{a priori} estimate for the
complexity of the simplified Hironaka algorithm. As a consequence of this 
result, we show that there exists canonical Hironaka embedded 
desingularization and principalization over fields of large characteristic 
(relative to the degrees of generating polynomials).
 \end{abstract}
\maketitle
\centerline{\it Dedicated to Professor Heisuke Hironaka on the occasion of his 80th birthday}

\bigskip

\tableofcontents
\addtocounter{section}{-1}

\section{Introduction} In this paper we discuss the complexity of  the Hironaka
theorem on resolution of singularities
of a marked ideal.  Recall that approach to the problem of embedded
resolution was originated by Hironaka (see \cite{Hir}) and later
developed and simplified  by Bierstone-Milman  (see \cite{B-M''},~\cite{BM},~\cite{BM1}), 
Villamayor (see \cite{Vi},~\cite{Vi1}), and W\l odarczyk (\cite{Wlod}) and others. In particular, we also use some elements from the recent development by Koll\'ar (\cite{Kollar2}).

It seems easier to estimate the complexity of the resolution algorithm from the
recursive descriptions in W\l odarczyk \cite{Wlod} or Bierstone-Milman
\cite{B-M5} than from the earlier iterative versions. The algorithms
in \cite{Wlod} and \cite{B-M5} (or \cite{BM2}) lead to identical blowing-up
sequences; whether one proof is preferable to the other is partly a matter of taste.
In this article, we estimate the complexity of the ``weak-strong desingularization''
algorithm (see Section 1) using the construction of \cite{Wlod}, though \cite{B-M5}
could also be used (see Remark in this section below). In a
subsequent paper, we plan to use \cite{B-M5} to give a comparable complexity
estimate for the algorithm of ``strong desingularization'' (where the centres of
blowing up are smooth subvarieties of the successive strict transforms).


The basic question which arises is in what terms to estimate 
the complexity \emph{a priori}? 
We recall (see, e.~g., \cite{Wagner}, \cite{G2}) that
the complexity is usually measured as a function of the bit-size of
the input. In particular, in this paper we study varieties and ideals
which are represented by families of polynomials with integer
coefficients, and the vector of all these coefficients (for an
initial variety and an ideal) is treated as an input.  Hironaka's
algorithm consists of many steps of elementary calculations, but
they are organized in several (nested) recursions where the
resolution of an object (a variety or a marked ideal, see below) is
reduced to resolutions of suitable objects with smaller values of
appropriate parameters (like dimension or multiplicity). It is
instructive to represent the Hironaka algorithm as a tree, to each 
node $a$ of which corresponds a marked ideal. The marked ideals which
correspond to child nodes of $a$ have either smaller multiplicity of an
ideal or smaller dimension of a variety. An initial marked ideal
corresponds to the root of the tree. The depth of the tree is
bounded by $2\cdot m$ where $m$ denotes the dimension of the initial
variety, while the number of the nested recursions does not exceed
$m+3$. It appears that just the number of nested recursions is
the overwhelming contribution to the complexity of the Hironaka's
algorithm.

That is why as a relevant language for expressing a complexity bound
we have chosen the Grzegorczyk class $\cE^l,\, l\ge 0$
\cite{Gr}, \cite{Wagner} which consists of (integer) functions whose
construction requires $l$ nested primitive recursions. The classes
$\cE^l,\, l\ge 0$ provide a hierarchy of the set of all
primitive-recursive functions $\cup_{l<\infty} \cE^l$. In
particular, $\cE^2$ contains all (integer) polynomials and
$\cE^3$ contains all finite compositions of the exponential
function. 

As an illustration of complexity bounds from small
Grzegorczyk classes, we give examples of a few algebraic-geometrical
computational problems: Polynomial factoring \cite{G2}, with
polynomial complexity (so in $\cE^2$); finding irreducible
components of a variety, with the exponential complexity (so in
$\cE^3$) \cite{G2}; and constructing a Groebner basis of an ideal,
with double-exponential complexity (so also in $\cE^3$)
\cite{Giusti, Mora}.

The principal complexity result of this paper (Theorem~\ref{main3}) states that the
complexity of resolution of an ideal on an $m$-dimensional variety is bounded by a
function from class $\cE^{m+3}$. We mention also that the complexity of 
Hilbert's Idealbasissatz for polynomial
ideals in $n$ variables (much simpler
from a purely mathematical point of view) belongs to class $\cE^{n+1}$ (cf. \cite{Seidenberg}, \cite{Simpson},
 where the latter was formulated in different languages), and, moreover, the
number $n+1$ is sharp. This shows that these two quite different algorithmic problems
have a common feature in recursion on the dimension which mainly determines their
complexities.

\begin{remark} The main differences between the proofs in \cite{B-M5} and
\cite{Wlod} come from the notions of derivative ideal that are used (\cite{B-M5}
uses only derivatives that preserve the ideal of the exceptional
divisor) and from passage to a ``homogenized ideal'' in \cite{Wlod} (see
\S2.8). The latter has the advantage that any two maximal contact hypersurfaces
for the homogenized ideal are related by an automorphism, while \cite{B-M5}
provides a stronger version of functoriality that is needed for strong
desingularization. Since \cite{B-M5} does not involve homogenization, certain
complexity estimates can be improved (see Remark  after Corollary \ref{number}),
although the overall Grzegorczyk complexity class $\cE^{m+3}$ is unchanged.
\end{remark}

We mention that in \cite{Vladuts, Teitelbaum, Kozen} polynomial complexity algorithms for
resolution of a curve are presented.
Also in \cite{Blanco}, a complexity estimated in $\cE^3$ for an algorithm for 
resolution of singularities in the
non-exceptional monomial case is presented.

In Section 1 below,
we formulate the results on canonical principalization of a
sheaf of ideals and on embedded desingularization. In Section 2 we give
definitions of basic notions 
like marked ideals, hypersurfaces of maximal contact and coefficient ideals,
and we formulate their properties (one can find proofs in \cite{Wlod}).
In Section 3 we describe the resolution algorithm. In Section 4, we provide bounds
on the degrees and on the number of polynomials which describe a single blow-up.
In Section 5 we give some auxiliary  bounds --- on the multiplicity of an 
ideal in terms of degrees of describing  polynomials, on the degree of a 
hypersurface of maximal contact, and on the number of generators 
of the coefficient ideal and their degrees.
Finally in Section 6 we estimate
the complexity of the resolution algorithm in terms of Grzegorczyk's classes (their definition is also provided
in Section 6). 

In the appendix we give some applications of the obtained estimates. We show that resolution of singularities exists in positive characteristic, provided that the characteristic is very large relatively to degree of polynomials describing singularities and their  number.

\subsection{Acknowledgements} We heartily thank the Max Planck Institut fuer 
Mathematik, Bonn for its warm hospitality.

\section{Formulation of the Hironaka  resolution  theorems}
All algebraic varieties in this paper are defined over a  ground field of characteristic zero. The assumption of characteristic zero is only needed for the local existence of a hypersurface of  maximal contact (Lemma \ref{le: Gi}).

We give proofs of the following Hironaka Theorems (see \cite{Hir}):
\begin{enumerate}

\item {\bf Canonical Principalization.}
\begin{theorem} \label{th: 1} Let ${\cI}$ be a sheaf of ideals on a smooth algebraic variety $X$.
There exists a principalization of ${\cI}$; that is, a sequence
$$ X=X_0 \buildrel \sigma_1 \over\longleftarrow X_1
\buildrel \sigma_2 \over\longleftarrow X_2\longleftarrow\ldots
\longleftarrow X_i \longleftarrow\ldots \longleftarrow X_r =\widetilde{X}$$
of blow-ups $\sigma_i:X_{i-1}\leftarrow X_{i}$ of smooth centers $C_{i-1}\subset
 X_{i-1}$,
such that:
\begin{enumerate}
\item The exceptional divisor $E_i$ of the induced morphism $\sigma^i=\sigma_1\circ \ldots\circ\sigma_i:X_i\to X$ has only  simple normal
crossings and $C_i$ has simple normal crossings with $E_i$.
\item The
total transform $\sigma^{r*}({\cI})$ is the ideal of a simple normal
crossing divisor
$\widetilde{E}$ which is  a natural  combination of the irreducible components of the divisor ${E_r}$.
\end{enumerate}
Moreover, the morphism $(\widetilde{X},\widetilde{\cI})\rightarrow(X,{\cI}) $ defined by 
the above principalization  commutes with smooth morphisms and  embeddings 
of ambient varieties. It is equivariant with respect to any group action,
not necessarily preserving the ground field $K$.
\end{theorem}

\item {\bf  Weak-Strong Hironaka Embedded  Desingularization.}

\begin{theorem} \label{th: emde} \label{th: 2} Let $Y$ be a subvariety of a smooth variety
$X$ over a field of characteristic zero.
There exists a  sequence $$ X_0=X \buildrel \sigma_1 \over\longleftarrow X_1
\buildrel \sigma_2 \over\longleftarrow X_2\longleftarrow\ldots
\longleftarrow X_i \longleftarrow\ldots \longleftarrow X_r=\widetilde{X}$$ of
blow-ups  $\sigma_i:X_{i-1}\longleftarrow X_{i}$ of smooth centers $C_{i-1}\subset
 X_{i-1}$, such that:
\begin{enumerate}
\item The exceptional divisor $E_i$ of the induced morphism $\sigma^i=\sigma_1\circ \ldots\circ\sigma_i:X_i\to X$ has only  simple normal
crossings and $C_i$ has simple normal crossings with $E_i$.
\item Let $Y_i\subset X_i$ be the strict transform of $Y$. All centers $C_i$ are disjoint from the set $\Reg(Y)\subset Y_i$ of points where   $Y$  (not $Y_i$) is smooth
(and are not necessarily contained in $Y_i$).
\item  The strict transform $\widetilde{Y}:=Y_r$ of  $Y$  is smooth and
has only simple normal crossings with the exceptional divisor $E_r$.
\item The morphism $(X,{Y})\leftarrow (\widetilde{X},\widetilde{Y})$ defined 
by the embedded desingularization commutes with smooth morphisms and  
embeddings of ambient varieties. It is equivariant with respect to any group 
action, not necessarily preserving the ground $K$.
 \end{enumerate}
\end{theorem}

\item {\bf Canonical Resolution of Singularities.}
\begin{theorem} \label{th: 3} Let $Y$ be an algebraic variety over a field of characteristic zero.
Then there exists a canonical desingularization of $Y$; that is,
a smooth variety $\widetilde{Y}$ together with a proper birational morphism $\res_Y: \widetilde{Y}\to Y$ such that:
\begin{enumerate}
\item $\res_Y: \widetilde{Y}\to Y$ is  an isomorphism over the nonsingular part of $Y$.
\item The inverse image $\res_Y^{-1}(Y_{\rm sing})$ of the singular locus of  
$Y$ is  a simple normal crossings divisor.
\item The morphism $\res_Y$  is functorial with respect to smooth morphisms.
For any smooth morphism $\phi: Y'\to Y$ there is a natural lifting $\widetilde{\phi}: \widetilde{Y'}\to\widetilde{Y}$ which is a smooth morphism.
\item $\res_Y$ is equivariant with respect to any group action, 
not necessarily preserving the ground field.
\end{enumerate}
\end{theorem}
\end{enumerate}

\begin{remark} Note that a blow-up of codimension one components  is an isomorphism. 
However it defines a nontrivial transformation of marked ideals. In the 
actual desingularization process, blow-ups of this kind may occur for 
some marked ideals induced on subvarieties of ambient varieties. Though they define isomorphisms of those subvarieties they determine blow-ups of ambient
varieties which are not isomorphisms.
\end{remark}

\begin{remarks} 
(1)\, By the exceptional divisor of a blow-up $\sigma: X'\to X$ with smooth 
center $C$ we mean the inverse image $E:=\sigma^{-1}(C)$ of the center C. By 
the exceptional divisor of a composite of blow-ups $\sigma_i$ with  smooth centers $C_{i-1}$ we mean the union of the strict transforms of the exceptional divisors of $\sigma_i$. This definition coincides with the standard definition of the exceptional set of points of the birational morphism in the case when $\codim(C_i)\geq 2$ (as in Theorem  \ref{th: 2}). If $\codim(C_{i-1})=1$ the blow-up of $C_{i-1}$ is an identical isomorphism and defines a formal operation of converting a subvariety $C_{i-1}\subset X_{i-1}$ into a component of the exceptional divisor $E_i$ on $X_i$. This formalism is convenient  for  the proofs. In particular it indicates that $C_{i-1}$ identified via $\sigma_i$ with  a component of $E_i$ has simple normal crossings with other components of $E_i$.
\smallskip

(2)\, In Theorem \ref{th: 2}, we blow up centers of codimension $\geq 2$ and both definitions coincide.
\end{remarks}

\section{Marked ideals, coefficient ideals and  hypersurfaces of maximal contact}
We shall assume that the ground field is algebraically closed.

\subsection{Resolution of marked ideals}

For any sheaf of ideals ${\cI}$ on  a smooth  variety $X$ and any point $x\in X$ we denote by
$$\ord_x({\cI}):=\max\{i\mid \cI\subset m_x^i\}$$ the {\it order} of ${\cI}$ at $x$.
(Here $m_x$ denotes the maximal ideal of $x$.)

\begin{definition}(Hironaka (see \Hi, \Hii),  Bierstone-Milman (see \BM),Villamayor (see \Vi))
A {\it marked ideal} (originally a {\it basic object} in Villamayor) is a collection $(X,{\cI},E,\mu)$, where $X$ is a smooth variety,
${\cI}$ is a  sheaf of ideals on $X$, $\mu$ is a nonnegative integer and $E$ is a totally ordered collection of  divisors
whose  irreducible components are pairwise disjoint and all have multiplicity one. Moreover the irreducible components of divisors in $E$ have simultaneously
simple normal crossings. \end{definition}
\begin{definition}(Hironaka (\Hi, \Hii), Bierstone-Milman (see \BM),Villamayor (see \Vi))
By the {\it support} (originally {\it singular locus}) of $(X,{\cI},E,\mu)$ we mean
$$\supp(X,{\cI},E,\mu):=\{x\in X\mid \ord_x(\cI)\geq \mu\}.$$
\end{definition}

\begin{remarks}
(1)\, Sometimes for simplicity we will represent marked ideals $(X,{\cI},E,\mu)$ as couples $({\cI},\mu)$ or even ideals ${\cI}$.

(2)\, For any sheaf of ideals ${\cI}$ on $X$, we have
 $\supp({\cI},1)=\supp({\cO_X / \cI})$.
 
(3)\, For any marked ideal $({\cI},\mu)$ on $X$,
$\supp({\cI},\mu)$ is a closed subset of $X$ (Lemma \ref{le: Vi1}).
\end{remarks}

\begin{definition}(Hironaka (see \Hi, \Hii),  Bierstone-Milman (see \BM),Villamayor (see \Vi))
By a {\it resolution} of $(X,{\cI},E,\mu)$ we mean
a sequence of blow-ups $\sigma_i:X_i\to X_{i-1}$ of smooth centers $C_{i-1}\subset
 X_{i-1}$,
$$
X_0=X\buildrel \sigma_1 \over\longleftarrow X_1
\buildrel \sigma_2 \over\longleftarrow X_2 \buildrel
\sigma_3 \over \longleftarrow\ldots
X_i\longleftarrow \ldots \buildrel \sigma_{r}  \over\longleftarrow X_r,$$
which defines a sequence of  marked ideals
$(X_i,{\cI}_i,E_i,\mu)$ where
\begin{enumerate}
\item $C_i
\subset
\supp(X_i,{\cI}_i,E_i,\mu)$.
\item $C_i$ has simple normal crossings with $E_i$.
\item ${\cI}_i=\cI(D_i)^{-\mu}\sigma_i^*({\cI}_{i-1})$, where $\cI(D_i)$ is the ideal of  the  exceptional divisor $D_i$
of $\sigma_i$.
\item $E_i=\sigma_i^{\rm c}(E_{i-1})\cup \{D_i\}$, where  $\sigma_i^{\rm c}(E_{i-1})$ is the set of strict transforms of divisors in $E_{i-1}$.
\item The order  on $\sigma_i^{\rm c}(E_{i-1})$ is defined by the order on 
$E_{i-1}$, while $D_i$ is the maximal element of $E_i$.
\item $\supp(X_r,{\cI}_r,E_r,\mu)=\emptyset$.
\end{enumerate}

\begin{definition}
A sequence of morphisms which are either isomorphisms or  blow-ups satisfying  conditions (1)-(5) is called a {\it multiple  blow-up}. The number of morphisms in a multiple blow-up will be called its {\it length}.
\end{definition}
\begin{definition}
An {\it extension} of a multiple blow-up (or a resolution) $(X_i)_{0\leq i\leq m}$ is a sequence $(X'_j)_{0\leq j\leq m'}$ of blow-ups and isomorphisms
$X'_{0}=X'_{j_0} =\ldots=X'_{j_1-1}\leftarrow X'_{j_1}=\ldots=X'_{j_2-1}\leftarrow \ldots X'_{j_m}=\ldots=X'_{m'}$,
where $X'_{j_i}=X_i$.
\end{definition}

\begin{remarks}
(1)\, The definition of  extension  arises naturally when we pass to open 
subsets of the ambient variety $X$.

(2)\, The notion of a {\it multiple  blow-up} is analogous to the notion of a 
sequence of {\it admissible} blow-ups considered by Hironaka, Bierstone-Milman and Villamayor.
\end{remarks}

\subsection{Transforms of marked ideal and controlled transforms of functions}
In the setting of the above definition we will call $$({\cI}_i,\mu):=\sigma_i^{{\rm c}}({\cI}_{i-1},\mu)$$
\end{definition}
\noindent the {\it transform of the marked ideal} or  {\it controlled transform} of $(\cI,\mu)$. It makes sense for a single blow-up in a multiple  blow-up as well as for a multiple  blow-up. Let  $\sigma^i:=\sigma_1\circ\ldots\circ\sigma_i: X_i\to X$ be a composition of consecutive morphisms of a multiple  blow-up. Then in the above setting
$$({\cI}_i,\mu)=\sigma^{i{\rm c}}({\cI},\mu).$$
We will also denote the controlled transform $\sigma^{i{\rm c}}({\cI},\mu)$ by
$(\cI,\mu)_i$ or $[\cI,\mu]_i.$

The controlled transform can also be defined for local sections $f\in\cI(U)$. 
Let $\sigma: X\leftarrow X'$ be a blow-up of a smooth center 
$C\subset \supp(\cI,\mu)$ defining transformation of marked ideals 
$\sigma^{\rm c}(\cI,\mu)=(\cI',\mu)$. Let $f\in \cI(U)$ be a section of $\cI$.
Let $U'\subseteq\sigma^{-1}(U)$ be an open subset for which the sheaf of ideals of the exceptional divisor is generated by a function $y$. The function $$g=y^{-\mu}(f\circ\sigma) \in  \cI(U')$$ \noindent
is the {\it controlled transform} of $f$ on $U'$ (defined up to an invertible function). As before we extend it  to any multiple  blow-up.

The following lemma shows that the notion of controlled transform is  well defined.
\begin{lemma} \label{bb} Let $C\subset \supp({\cI},\mu)$ be a smooth center of 
a blow-up $\sigma: X\leftarrow X'$ and let $D$ denote the exceptional divisor.  Let ${\cI}_C$ denote the sheaf of ideals defined by $C$. Then
\begin{enumerate}
\item $\cI \subset {\cI}_C^\mu$.
\item  $\sigma^*({\cI})\subset ({\cI}_D)^\mu$.
\end{enumerate}
\end{lemma}

\begin{proof} (1)\, We can assume that the ambient variety $X$ is affine. Let $u_1,\ldots,u_k$ be parameters generating ${\cI}_C$
Suppose $f \in \cI \setminus {\cI}_C^\mu$. Then we can write $f=\sum_{\alpha}c_\alpha u^\alpha$, where either $|\alpha|\geq \mu$ or  $|\alpha|< \mu$ and $c_\alpha\not\in {\cI}_C$.
By the assumption there is $\alpha$ with $|\alpha|< \mu$ such that $c_\alpha\not \in {\cI}_C$. Take  $\alpha$  with  the smallest $|\alpha|$. There is a point $x\in C$ for which $c_\alpha(x)\neq 0$ and in the Taylor expansion of $f$ at $x$ there is a term $c_\alpha(x) u^\alpha$. Thus $\ord_x(\cI)< \mu$. This contradicts to the assumption $C\subset \supp({\cI},\mu)$.

(2)\, $\sigma^*({\cI})\subset \sigma^*({\cI}_C)^\mu=({\cI}_D)^\mu$.
\end{proof}

\subsection{Hironaka resolution principle}

Our proof is based on the following principle which can be traced back to Hironaka and was used by Villamayor in his simplification of Hironaka's algorithm:
\begin{eqnarray}  &\textbf{  (Canonical) Resolution of marked ideals $(X,\cI,E,\mu)$}\\ &\textbf{$\Downarrow$}\nonumber\\
& \textbf{ (Canonical) Principalization of the sheaves $\cI$ on $X$}\\
&\textbf{$\Downarrow$}\nonumber\\
&\textbf{(Canonical) Weak Embedded
Desingularization of subvarieties $Y\subset X$}\\
&\textbf{$\Downarrow$}\nonumber\\
&\textbf{(Canonical) Desingularization} \end{eqnarray}

\noindent
(1)$\Rightarrow$(2).  It follows immediately from the definition that a resolution of $(X,{\cI},\emptyset,1)$
determines a principalization of ${\cI}$. Denote by $\sigma: X\leftarrow \widetilde{X}$ the morphism defined by a resolution of $(X,\cI,\emptyset,1)$. The controlled transform $(\widetilde{\cI},1):=\sigma^{\rm c}(\cI,1)$ has the empty support. Consequently, $V(\widetilde{\cI})=\emptyset$,  and thus $\widetilde{\cI}$ is equal to the structural sheaf $\cO_{\widetilde{X}}$. This implies that the full transform $\sigma^*(\cI)$ is principal and generated by the sheaf of ideal of a divisor whose components are the exceptional divisors.
The actual process of desingularization is often achieved before $(X,{\cI},E,1)$ has been resolved (see \cite{Wlod}).

\noindent
(2)$\Rightarrow$(3).
Let $Y\subset X$ be an
irreducible  subvariety. Assume there is a principalization of sheaves of ideals ${\cI}_Y$ subject to conditions (a) and (b) of Theorem \ref{th: 1}.
Then, in the course of the
principalization of ${\cI}_Y$, the strict transform $Y_i$
of $Y$ in some $X_i$ is the center of a blow-up. At this stage $Y_i$ is nonsingular and has simple normal crossings with the exceptional divisors.

\noindent
(3)$\Rightarrow$(4). Every algebraic variety locally admits an embedding into an affine space.  Then we can show that the existence of canonical embedded desingularization independent of the embedding defines a canonical desingularization.

For more details, see \cite{Wlod}.

\begin{remark} {\bf Resolution scheme and marked ideals}.
Marked ideals will be understood as objects which carry vital information in the resolution scheme. There are four different types of information that can be associated with marked ideals:
\begin{enumerate}
\item The support $\supp(\cI,\mu)$ is the ``bad locus" which shall be eliminated. The blow-ups performed should have centers inside of $\supp(\cI,\mu)$.
\item The controlled transform $\sigma^c(\cI,\mu)=\cI_D^{-\mu}\sigma^{{*}}({\cI},\mu)
$ is the transform of the marked ideal associated with blow-ups with centers inside 
$\supp(\cI,\mu)$.
\item The resolution of $\supp(\cI,\mu)$ is the sequence of blow-ups and the induced transformations of marked ideals eliminating the 
support of  the resulting marked ideal $(\cI,\mu)$.
\item Canonical resolution is a unique resolution which will be assigned to a marked ideal.
Once we assign to a certain class of marked ideals their canonical resolutions they become useful operations  to resolve some larger class of marked ideals. In other words, the resolution of a certain marked ideal is always reduced to resolution of some ``simpler" marked ideals.
The notion of simplicity refers essentially to two very rough invariants : the dimension of the ambient variety, and the order of nonmonomial part.
\end{enumerate}
The algorithm builds upon two different canonical reductions:
\begin{itemize}
\item reduction of order by resolving a so called ``companion ideal" (see Step 2 in Section \ref{Resolution algorithm}).
\item reduction of dimension of the ambient variety which relies on the two fundamental concepts
of hypersurface of  {\it maximal contact}, and {\it coefficient ideal} (see Sections \ref{Maximal}, \ref{Coefficient}).
\end{itemize}
\end{remark}

\subsection{Equivalence relation for marked ideals}

Let us introduce the following equivalence relation for marked ideals:
\begin{definition} Let  $(X,{\cI},E_{\cI},\mu_{\cI})$ and $(X,{\cJ},E_{\cJ},\mu_{\cJ})$ be two marked ideals
on a smooth variety $X$. Then
{$(X,{\cI},E_{\cI},\mu_{\cI})\simeq (X,{\cJ},E_{\cJ},\mu_{\cJ})$}
\noindent if:
\begin{enumerate}
\item $E_{\cI}=E_{\cJ}$ and the orders on $E_{\cI}$ and on $E_{\cJ}$ coincide.
\item $\supp({\cI},\mu_{\cI})=\supp({\cJ},\mu_{\cJ}).$
\item The multiple  blow-ups $(X_i)_{i=0,\ldots,k}$ are the same for both marked ideals, and
$\supp({\cI}_i,\mu_{\cI})=\supp({\cJ}_i,\mu_{\cJ}).$
\end{enumerate}
\end{definition}

\begin{example}
For any $k\in {\bf N}$, $({\cI},\mu)\simeq ({\cI}^k,k\mu)$.
\end{example}

\begin{remark} The marked ideals considered in this paper satisfy a stronger equivalence condition:
For any smooth morphism $\phi: X' \to X$, $\phi^*(\cI,\mu)\simeq\phi^*(\cJ,\mu)$.
This condition will follow and is not added in the definition.
 \end{remark}


\subsection{Ideals of derivatives \label{der}}
Ideals of derivatives were first introduced and studied in the resolution context by Giraud.
Villamayor developed and applied this language to his {\it basic objects}.

\begin{definition}(Giraud, Villamayor) Let ${\cI}$ be a coherent sheaf of ideals on a smooth variety $X$. By the {\it  first derivative} (originally {\it extension}) ${\cD}({\cI})$ of $\cI$ we mean the coherent sheaf of ideals generated by all functions
$f\in {\cI}$ together with their first derivatives. Then the {\it i-th derivative} ${\cD}^i({\cI})$ is defined to be ${\cD}({\cD}^{i-1}({\cI}))$. If $({\cI},\mu)$ is a marked ideal and $i\leq \mu$ then we define
$${\cD}^i({\cI},\mu):=({\cD}^i({\cI}),\mu-i).$$
\end{definition}

Recall that on a smooth variety $X$ there is a locally free sheaf of differentials $\Omega_{X/K}$ over $K$ generated locally by $du_1,\ldots, du_n$ for a set of local parameters $u_1,\ldots, u_n$. The dual sheaf of derivations $\Der_K(\cO_X)$ is locally generated by the derivations $\frac{\partial }{\partial u_i}$.
Immediately from the definition we observe that ${\cD}({\cI})$ is a coherent sheaf
defined locally by  generators $f_j$ of ${\cI}$ and all their partial derivatives $\frac{\partial f_j}{\partial u_i}$. We see by induction that ${\cD}^i({\cI})$ is a coherent sheaf
defined locally by the generators $f_j$ of ${\cI}$ and their derivatives $\frac{\partial^{|\alpha|} f_j}{\partial u^\alpha}$ for all multiindices $\alpha=(\alpha_1,\ldots,\alpha_n)$, where $|\alpha|:=\alpha_1+\ldots+\alpha_n\leq i$.

\begin{lemma}(Giraud, Villamayor) \label{le: Vi1}
For any  $i\leq\mu-1$,  $$\supp({\cI},\mu)=\supp({\cD}^i({\cI}),\mu-i).$$
In particular,  $\supp({\cI},\mu)=\supp({\cD}^{\mu-1}({\cI}),1)=V({\cD}^{\mu-1}({\cI}))$ is a closed set.
\end{lemma}

We write $({\cI},\mu)\subset ({\cJ},\mu)$ if $\cI\subset {\cJ}$.

\begin{lemma}(Giraud,Villamayor) \label{le: inclusions} Let $({\cI},\mu)$ be a marked ideal,
$C\subset\supp ({\cI},\mu)$ a smooth center, and $r\leq \mu$.
Let $\sigma: X\leftarrow
X'$ be a blow-up at $C$. Then $$\sigma^{\rm c}({\cD}^r({\cI},\mu)) \subseteq {\cD}^r(\sigma^{\rm c}({\cI},\mu)).$$

\end{lemma}

\begin{proof}  See the simple computations in \cite{Vi2}, \cite{Wlod}.
\end{proof}

\subsection{Hypersurfaces of maximal contact}\label{Maximal}

The concept of the {\it hypersurfaces of maximal contact} is one of the key points of this proof. It was
originated by Hironaka, Abhyankhar and Giraud and developed in the papers of Bierstone-Milman and Villamayor.
In our terminology, we are looking for a smooth hypersurface containing the support of a marked ideal and whose strict transforms under multiple  blow-ups contain the supports of the induced marked ideals. Existence of such hypersurfaces allows a reduction of the resolution problem to  codimension 1.

First we introduce marked ideals which locally admit  hypersurfaces of maximal contact.

\begin{definition}(Villamayor (see \cite{Vi}) \label{Vi1})
We say that a marked ideal $({\cI},\mu)$  is of {\it maximal order} (originally {\it simple basic object}) if $\max\{\ord_x({\cI})\mid x\in X\}\leq \mu$ or equivalently ${\cD}^\mu({\cI})=\cO_X$.
\end{definition}

\begin{lemma}(Villamayor (see \cite{Vi}) \label{Vi2}) Let $({\cI},\mu)$ be a marked ideal of maximal order
and let $C\subset\supp ({\cI},\mu)$ be a smooth center. Let $\sigma: X\leftarrow
X'$ be a blow-up at $C$. Then $\sigma^{\rm c}({\cI},\mu)$ is of maximal order.
\end{lemma}

\begin{proof} If $({\cI},\mu)$ is a marked ideal of maximal order then ${\cD}^\mu({\cI})=\cO_X$.
Then, by  Lemma \ref{le: inclusions},  ${\cD}^\mu(\sigma^{\rm c}({\cI},\mu))\supset \sigma^{\rm c}({\cD}^\mu({\cI}),0)=\cO_X$. 
\end{proof}

\begin{lemma}(Villamayor (see \cite{Vi}),  \label{Vi3}) If $({\cI},\mu)$ is a marked ideal of maximal order and $0\leq i \leq \mu$, then ${\cD}^{i}({\cI},\mu)$
is of maximal order.
\end{lemma}

\begin{proof} ${\cD}^{\mu-i}({\cD}^{i}({\cI},\mu))={\cD}^{\mu}({\cI},\mu)=\cO_X$.
\end{proof}

\begin{lemma}(Giraud (see \cite{G}))\label{le: Gi}  Let $({\cI},\mu)$ be a marked ideal of maximal order.  Let $\sigma: X\leftarrow
X'$ be a blow-up at a smooth center $C\subset\supp({\cI},\mu)$. Let $u\in {\cD}^{\mu-1}({\cI},\mu)(U)$ be a function such that, for any $x\in V(u)$, $\ord_x(u)=1$.
Then \begin{enumerate}
\item $V(u)$ is smooth;
\item $\supp({\cI},\mu)\cap U \subset V( u).$
\end{enumerate}
Let $U'\subset \sigma^{-1}(U)\subset X'$ be an open
set where the exceptional divisor is described by $y$.  Let $u':=\sigma^{\rm c}(u)=y^{-1}\sigma^*(u)$ be the controlled transform of $u$.
Then
\begin{enumerate}
\item $u'\in {\cD}^{\mu-1}(\sigma^{\rm c}({\cI}_{|U'},\mu));$
\item  $V(u')$ is smooth;
\item $\supp({\cI'},\mu)\cap U' \subset V( u')$
\item $V(u')$ is the restriction of the strict transform of $V(u)$ to $U'$.
\end{enumerate}
\end{lemma}

\begin{proof} (1) $u'=\sigma^{\rm c}(u)=u/y\in\sigma^{\rm c}({\cD}^{\mu-1}({\cI}))\subset {\cD}^{\mu-1}(\sigma^{\rm c}({\cI}))$.

(2) Since $u$ is one of the local parameters describing the center of the blow-up, $u'=u/y$ is a parameter; that
is, a function of order one.

 (3) follows from (2). 
 \end{proof}
 
\begin{definition} We will call a function $$u\in T({\cI})(U):={\cD}^{\mu-1}({\cI}(U))$$ \noindent of multiplicity one a {\it tangent direction} of $({\cI},\mu)$ on $U$.
\end{definition}

As a corollary from the above we obtain the following lemma.

\begin{lemma}(Giraud)\label{le: Giraud}
Let $u\in T({\cI})(U)$ be a tangent direction of $({\cI},\mu)$ on $U$. Then  for any multiple  blow-up  $(U_i)$ of $({\cI}_{|U},\mu)$, all the supports  of the induced marked ideals  $\supp(\cI_i,\mu)$ are contained in the strict transforms $V(u)_i$ of $V(u)$. 
\end{lemma}

\begin{remarks}
(1)\, Tangent directions are functions locally defining hypersurfaces of maximal contact.

(2)\, The main problem leading to complexity of the proofs is that of noncanonical choice of
the tangent directions. We overcome this difficulty by introducing {\it homogenized ideals}.
\end{remarks}

\subsection{Arithmetical operations on marked ideals}
In this sections all marked ideals are defined for the smooth variety $X$ and the same
set of exceptional divisors $E$.
Define the following operations of addition and multiplication of marked ideals:
\begin{enumerate}
\item $({\cI},\mu_{\cI})+({\cJ},\mu_{\cJ}):=({\cI}^{\mu_{\cI}}+{\cJ}^{\mu_{\cI}},\mu_{\cI}\mu_{\cJ}),$
or, more generally, 
$$({\cI}_1,\mu_1)+\ldots+({\cI}_m,\mu_m):=({\cI}_1^{\mu_2\cdot\ldots\cdot\mu_m}+
{\cI}_2^{\mu_1\mu_3\cdot\ldots\cdot\mu_m}+\ldots+{\cI}_m^{\mu_1\ldots\mu_{k-1}} ,\mu_1\mu_2\ldots\mu_m)$$
(the operation of addition is not associative).
\item $(\cI,\mu_\cI) \cdot (\cJ,\mu_\cJ):=(\cI\cdot J,\mu_\cI+\mu_\cJ)$.
\end{enumerate}

\begin{lemma} \label{le: operations}
(1)\,
$\supp(({\cI}_1,\mu_1)+\ldots+({\cI}_m,\mu_m))= \supp({\cI}_1,\mu_1)\cap\ldots\cap\supp({\cI}_m,\mu_m)$. Moreover,  multiple  blow-ups $(X_k)$ of  $({\cI}_1,\mu_1)+\ldots+({\cI}_m,\mu_m)$ are exactly those which are simultaneous multiple  blow-ups for all $({\cI}_j,\mu_j)$,  and, for any $k$, we have the following equality for the controlled transforms $({\cI}_j,\mu_\cI)_k$: $$({\cI}_1,\mu_1)_k+\ldots+({\cI}_m,\mu_m)_k=[({\cI}_1,\mu_1)+\ldots+({\cI}_m,\mu_m)]_k.$$

(2)\,
$\supp(\cI,\mu_\cI)\cap\supp(\cJ,\mu_\cJ)\supseteq\supp((\cI,\mu_\cI)\cdot (\cJ,\mu_\cJ))$.
Moreover, any simultaneous multiple  blow-up $X_i$ of  both ideals $(\cI,\mu_\cI)$ and $(\cJ,\mu_\cJ)$ is a multiple  blow-up for
$(\cI,\mu_\cI)\cdot (\cJ,\mu_\cJ)$, and for the controlled transforms $(\cI_k,\mu_\cI)$ and $(\cJ_k,\mu_\cJ)$, we have the equality
$$(\cI_k,\mu_\cI)\cdot(\cJ_k,\mu_\cJ)=[(\cI,\mu_\cI)\cdot(\cJ,\mu_\cJ)]_k.$$
\end{lemma}

\subsection{Homogenized ideals and tangent directions\label{hom}}
Let $({\cI},\mu)$ be a marked ideal of maximal order. Set $T({\cI}):={\cD}^{\mu-1}{\cI}$.
By the {\it homogenized ideal} we mean $${\cH}({\cI},\mu):=({\cH}({\cI}),\mu)=({\cI}+{\cD}\cI\cdot T({\cI})+\ldots+{\cD}^i\cI\cdot T({\cI})^i+ \ldots+{\cD}^{\mu-1}\cI\cdot T({\cI})^{\mu-1},\mu)$$

\begin{remark}
A homogenized ideal has two important properties:
\begin{enumerate}
\item It is equivalent to the given ideal.
\item It ``looks the same'' from all possible tangent directions.
\end{enumerate}
By the first property we can   use the homogenized ideal to construct resolution via the Giraud Lemma \ref{le: Giraud}. By the second property such a construction does not depend on the choice of tangent directions.
\end{remark}

\begin{lemma} Let $({\cI},\mu)$ be a marked ideal of maximal order. Then
\begin{enumerate}
\item $({\cI},\mu)\simeq ({\cH}({\cI}),\mu)$ (see Definition 2.4.1).
\item For any multiple  blow-up $(X_k)$ of $(\cI,\mu)$, $$({\cH}({\cI}),\mu)_k=({\cI},\mu)_k+[{\cD}(\cI,\mu)]_k\cdot [(T({\cI}),1)]_k+\ldots [{\cD}^{\mu-1}(\cI,\mu)]_k\cdot +[(T({\cI}),1)]_k^{\mu-1}.$$
\end{enumerate}
\end{lemma}

Although the following Lemmas  \ref{le: homo0} and \ref{le: homo} are used in this paper only in the case $E=\emptyset$, we formulate them in slightly more general versions.

\begin{lemma} \label{le: homo0} Let $(X,{\cI},E,\mu)$ be a marked ideal of maximal order.
Assume there exist   tangent directions $u,v\in T(\cI,\mu)_x={\cD}^{\mu-1}({\cI},\mu)_x$ at $x\in\supp({\cI},\mu)$ which are transversal to $E$.
Then there exists an automorphism $\widehat{\phi}_{uv}$ of the completion $\widehat{X}_x:=\Spec(\widehat{\cO}_{x,X})$ such that
\begin{enumerate}
\item  $\widehat{\phi}_{uv}^*({\cH}\widehat{\cI})_x=({\cH}\widehat{\cI})_x$;
\item  $\widehat{\phi}_{uv}^*(E)=E$;
\item  $\widehat{\phi}_{uv}^*(u)=v$;
\item  $\supp(\widehat{\cI},\mu):=V(T(\widehat{\cI},\mu))$ is contained in the fixed point set of $\phi$.
\end{enumerate}
\end{lemma}

\begin{proof} (0) {\bf Construction of the automorphism $\widehat{\phi}_{uv}$.}
Find parameters $u_2,\ldots, u_n$ transversal to $u$ and $v$ such that
 $u=u_1,u_2,\ldots, u_n$  and $v,u_2,\ldots, u_n$ form  two sets of parameters at $x$ and divisors in $E$ are described by some parameters $u_i$ where $i\geq 2$.
Set \begin{equation}\widehat{\phi}_{uv}(u_1)=v, \quad \widehat{\phi}_{uv}(u_i)=u_i \quad\mbox{for} \quad i>1.\nonumber \end{equation}

(1)  Let $h:=v-u\in {T({\cI})}$. For any  $f\in \widehat{\cI}$,
 \begin{equation} \widehat{\phi}_{uv}^*(f)=f(u_1+h,u_2,\ldots,u_n)= f(u_1,\ldots,u_n)+\frac{\partial{f}}{\partial{u_1}}\cdot h+ \frac{1}{2!}
 \frac{\partial^2{f}}{\partial{u_1^2}}\cdot h^2+\ldots +\frac{1}{i!}
 \frac{\partial^i{f}}{\partial{u_1^i}}\cdot h^i +\ldots\nonumber \end{equation}
 The latter element belongs to \begin{equation}\widehat{\cI}+{\cD}\widehat{\cI}\cdot \widehat{T({\cI})}+\ldots +{\cD}^i\widehat{\cI}\cdot \widehat{T({\cI})}^i+ \ldots +{\cD}^{\mu-1}\widehat{\cI}\cdot \widehat{T({\cI})}^{\mu-1}={\cH}\widehat{\cI}.\nonumber\end{equation}
Hence $\widehat{\phi}_{uv}^*(\widehat{\cI})\subset {\cH}\widehat{\cI}$.

(2)(3) follow from the construction.

(4) The fixed point set of $\widehat{\phi}^*_{uv}$ is defined by $u_i=\widehat{\phi}^*_{uv}(u_i)$, $i=1,\ldots,n$;  that is, by $h=0$. But $h\in {\cD}^{\mu-1}({\cI})$ is $0$ on $\supp({\cI},\mu)$. 
\end{proof}

\begin{lemma} \label{le: homo} {\bf Glueing Lemma.} Let $(X,{\cI},E,\mu)$ be a marked ideal of maximal order for which there exist
tangent directions $u,v\in T({\cI},\mu)$ at  $x\in\supp({\cI},\mu)$ which are transversal to $E$.
Then there exist \'etale neighborhoods $\phi_{u},\phi_v: \overline{X}\to X$ of $x=\phi_u(\overline{x})=\phi_v(\overline{x}) \in X$,  where $\overline{x}\in \overline{X}$, such that
\begin{enumerate}
\item  $\phi_{u}^*({\cH}({\cI}))=\phi_{v}^*({\cH}({\cI}))$;
\item  $\phi_{u}^*(E)=\phi_{v}^*(E)$;
\item  $\phi_{u}^*(u)=\phi_{v}^*(v)$.
\end{enumerate}
Set $(\overline{X},\overline{\cI},\overline{E},\mu):=\phi_{u}^*(X,{\cH}({\cI}),E,\mu)=\phi_{v}^*(X,{\cH}({\cI}),E,\mu).$
\begin{enumerate}
\item[(4)]  For any $\overline{y}\in \supp(\overline{X},\overline{\cI},\overline{E},\mu)$, $\phi_u(\overline{y})=\phi_v(\overline{y})$.
\item[(5)]  For any multiple  blow-up $(X_i)$ of $({X},{\cI},\emptyset,\mu)$, the induced multiple  blow-ups $\phi_u^*(X_i)$ and $\phi_v^*(X_i)$ of $(\overline{X},\overline{\cI},\overline{E},\mu)$ are the same (defined by the same centers).
\end{enumerate}
Set $(\overline{X}_i):=\phi_u^*(X_i)=\phi_v^*(X_i)$.
\begin{enumerate}
\item[(6)] For any  $\overline{y}_i\in \supp(\overline{X}_i,\overline{\cI}_i,\overline{E}_i,\mu)$,
 $\phi_{u i}(\overline{y}_i)=\phi_{v i}(\overline{y}_i)$, where 
$\phi_{u i},\phi_{v i}:\overline{X}_i\to X_i$ are the induced morphisms.
\end{enumerate}
\end{lemma}

\begin{proof}
(0) {\bf Construction of  \'etale neighborhoods  ${\phi}_{u}, {\phi}_{v}: U\to X$.}
Let $U\subset X$ be an open subset for which there exist $u_2,\ldots, u_n$ which are transversal to $u$ and $v$ on $U$ such  that
 $u=u_1,u_2,\ldots, u_n$  and $v,u_2,\ldots, u_n$ form  two sets of parameters on $U$, and the divisors in $E$ are described by some $u_i$, where $i\geq 2$.
 Let ${\bf A}^n$ be the affine space with coordinates $x_1,\ldots,x_n$.
First  construct \'etale morphisms $\phi_1,\phi_2: U\to {\bf A}^n$ with
 \begin{equation} {\phi}^*_{1}(x_i)=u_i \quad \textrm{for all } i \quad\quad \mbox{and} \quad
{\phi}^*_{2}(x_1)=v, \quad {\phi}^*_{2}(x_i)=u_i\quad\mbox{for} \quad i>1.\nonumber \end{equation}
Then \begin{equation} \overline{X}:=U\times_{{\bf A}^n}U \nonumber \end{equation} \noindent is a fiber product for  the morphisms $\phi_1$ and $\phi_2$. The morphisms $\phi_u$, $\phi_v$ are defined to be the natural projections $\phi_u, \phi_v: \overline{X}\to U$ such that $\phi_1\phi_u=\phi_2\phi_v$. Set
\begin{align*} w_1&:=\phi_u^*(u)=(\phi_1\phi_u)^*(x_1)=(\phi_2\phi_v)^*(x_1)=\phi_v^*(v), \\
w_i&:=\phi_u^*(u_i)=\phi_v^*(u_i),\quad  \textrm{for $i\geq 2$}. 
\end{align*}
 
(1), (2), (3) follow from the construction.

(4) Let $h:=v-u$. By the above the morphisms $\phi_u$ and $\phi_v$ coincide on $\phi_u^{-1}(V(h))=\phi_v^{-1}(V(h))$.

By (4), a blow-up of a center $C\subset \supp(\cH(\cI))$  lifts to the  blow-ups at the
same center $\phi_u^{-1}(C)=\phi_v^{-1}(C)$. Thus (5), (6) follow (see \cite{Wlod}
for details).
\end{proof}

\subsection{Coefficient ideals and Giraud Lemma}\label{Coefficient}
The idea of a coefficient ideal was originated by Hironaka and then developed in papers of Villamayor and Bierstone-Milman.

\begin{example} {\bf Motivating example.} Assume that $u=0$ defines locally a hypersurface of maximal contact.
Consider a coordinate system $u=u_1,u_2,\ldots,u_n$. Write any function $f\in (\cI,\mu)$ as follows
$$ f:=c_{\mu,f}\cdot u^{\mu}+ c_{\mu-1,f}(u_2,\ldots,u_n)u^{\mu-1}+\ldots+c_{0,f}((u_2,\ldots,u_n)$$
Then it can be  easily seen that
$$\ord_x(f)\geq \mu \quad \Leftrightarrow \quad \ord_x(c_{\mu-i,f})\geq \mu-i\quad\mbox{for}\quad\mbox{all}\quad i=1,\ldots,\mu.$$
In other words, $$\supp(\cI,\mu)=\supp({\rm Coeff}_{V(u)}(\cI,\mu)),$$
where $${\rm Coeff}_{V(u)}(\cI,\mu):= \{((c_{\mu-i,f|V(u)})_{f\in \cI},\mu-i)| i=1,\ldots,\mu\}.$$
Here ${\rm Coeff}_{V(u)}(\cI,\mu)$ can be considered as a very first definition of coefficient ideal. It allows one to reduce resolution of $(\cI,\mu)$ to a resolution of ${\rm Coeff}_{V(u)}(\cI,\mu)$ ``living" on a  hypersurface of maximal contact.
One of the problems here is that this definition depends on a choice of coordinates. That is why we replace it with the definition below.

It is important to observe that the controlled transformed preserves the  form above:
$$ \sigma^c(f,\mu):=c_{\mu,f}\cdot u'^{\mu}+ c'_{\mu-1,f}(u'_2,\ldots,u'_n)u'^{\mu-1}+\ldots+c'_{0,f}((u'_2,\ldots,u'_n),$$
where $u'=\sigma^c(u,1), c'_{\mu-i,f}=\sigma^c(c_{\mu-i,f},\mu-i).$
In other words
$$\supp (\sigma^c(\cI,\mu))=\supp (\sigma^c({\rm Coeff}_{V(u)}(\cI,\mu))).$$
\end{example}

The following definition modifies and generalizes the definition of Villamayor.

\begin{definition}
Let $({\cI},\mu)$ be a marked ideal of maximal order. By the
 {\it coefficient ideal}  we mean
$${\cC}({\cI},\mu)=\sum_{i=1}^\mu ({\cD}^i{\cI},\mu-i).$$
\end{definition}

\begin{remark} The coefficient ideal $\cC(\cI)$ has two important properties:
\begin{enumerate}
\item  $\cC(\cI)$ is equivalent to  $\cI$.
\item The intersection  of the support of $(\cI,\mu)$ with any smooth subvariety $S$ is the support of the restriction of $\cC(\cI)$ to $S$:
$$\supp(\cI)\cap S=\supp(\cC(\cI)_{|S}).$$
\noindent Moreover this condition is persistent under relevant multiple  blow-ups.
\end{enumerate}
These properties allow one to control and modify the part of support of $(\cI,\mu)$ contained in $S$ by applying multiple  blow-ups of  $\cC(\cI)_{|S}$.
\end{remark}

\begin{lemma}\label{le: coeff}
 ${\cC}({\cI},\mu)\simeq({\cI},\mu)$.
\end{lemma}

\begin{proof} By Lemma \ref{le: operations}, multiple  blow-ups of $\cC(\cI,\mu)$ are simultaneous multiple  blow-ups of $\cD^i(\cI,\mu)$ for $0\leq i\leq \mu-1$. By Lemma \ref{le: inclusions}, multiple  blow-ups of $(\cI,\mu)$ define  multiple  blow-ups of all $\cD^i(\cI,\mu)$.
Thus multiple  blow-ups of $(\cI,\mu)$ and $\cC(\cI,\mu)$ are the same and $\supp({\cC}({\cI},\mu))_k=\bigcap \supp({\cD}^i{\cI},\mu-i)_k=\supp({\cI}_k,\mu).$
\end{proof}

\begin{lemma} \label{le: S} Let $(X,{\cI},E,\mu)$ be a marked ideal of maximal order.  Assume that $S$ has only simple normal crossings with $E$. Then
$$\supp({\cI},\mu)\cap S= \supp({\cC}({\cI},\mu)_{|S}).$$
Moreover  let  $(X_i)$ be  a multiple blow-up with  centers $C_i$   contained in the strict transforms $S_i\subset X_i$ of $S$. Then:
\begin{enumerate}
\item The restrictions  $\sigma_{i|S_i}: S_i\to S_{i-1}$ of the morphisms $\sigma_i: X_i\to X_{i-1}$ define
  a multiple  blow-up $(S_i)$
of ${\cC}({\cI},\mu)_{|S}$.
\item $\supp({\cI}_i,\mu)\cap S_i= \supp[{\cC}({\cI},\mu)_{|S}]_i.$
\item Every multiple  blow-up $(S_i)$  of ${\cC}({\cI},\mu)_{|S}$  defines  a multiple  blow-up $(X_i)$ of $({\cI},\mu)$ with  centers $C_i$ contained in the strict transforms $S_i\subset X_i$ of $S\subset X$.
\end{enumerate}
\end{lemma}

\begin{proof}
By Lemma \ref{le: coeff}, $\supp({\cI},\mu)\cap S= \supp ({\cC}({\cI},\mu))\cap S\subseteq\supp({\cC}({\cI},\mu)_{|S})
$.
Let $x_1,\ldots,x_k,y_1,\ldots, y_{n-k}$ be local parameters at $x$ such that $\{x_1=0,\ldots,x_k=0\}$ describes $S$. Then any function $f\in {\cI}$ can be written as
$$f=\sum c_{\alpha f}(y) x^\alpha,$$\noindent where $c_{\alpha f}(y)$ are formal power series in $y_i$.

Now $x\in \supp({\cI},\mu)\cap S$ iff $\ord_x(c_\alpha)\geq \mu-|\alpha|$ for all $f\in {\cI}$ and $|\alpha|\leq \mu$. Note that $$c_{\alpha f|S}=\bigg(\frac{1}{\alpha!}\frac{\partial^{|\alpha|}(f)}{\partial x^\alpha}\bigg)_{|S}\in {\cD}^{|\alpha|}({\cI})_{|S}$$ and consequently  $\supp({\cI},\mu)\cap S=\bigcap_{f\in {\cI}, |\alpha|\leq \mu}\supp({c_{\alpha f|S}}, \mu-|\alpha|)\supset  \supp({\cC}({\cI},\mu)_{|S})$.

The above relation is preserved by multiple  blow-ups of $({\cI},\mu)$. For details
see \cite{Wlod}.
\end{proof}

\begin{lemma} \label{le: etale3} Let $\phi: X'\to X$ be an \'etale morphism of smooth varieties
and let $(X,{\cI},\emptyset,\mu)$ be a marked ideal. Then
\begin{enumerate}
\item $\phi^*({\cD}({\cI}))={\cD}(\phi^*({\cI}))$;
\item $\phi^*({\cH}({\cI}))={\cH}(\phi^*({\cI}))$;
\item $\phi^*({\cC}({\cI}))={\cC}(\phi^*({\cI}))$.
\end{enumerate}
\end{lemma}

\begin{proof}
 Note that for any point $x\in X$ the completion $\widehat{\phi_x^*}$ is an isomorphism.
Thus $\widehat{\phi_x^*}\widehat{({\cD}({\cI}))}=\widehat{{\cD}(({\cI}))}$ and therefore $\phi^*({\cD}({\cI}))={\cD}(\phi^*({\cI})).$ (2) and (3) follow from (1).
\end{proof}


\section{Resolution algorithm}\label{Resolution algorithm}

The presentation of the following Hironaka resolution algorithm  builds upon
 Bierstone-Milman's, Villamayor's and Wlodarczyk's algorithms which are simplifications of the original Hironaka proof. We also use Koll\'ar's trick allowing to completely eliminate the use of invariants.

\begin{remarks} (1)\, Note that a blow-up of codimension one components  is an isomorphism. However it defines a nontrivial transformation of marked ideals. The inverse image of the center is still called the exceptional divisor.

(2)\, In the actual desingularization process, this kind of blow-up may occur for some marked ideals induced on subvarieties of ambient varieties. Though they define isomorphisms of these subvarieties, they determine blow-ups of ambient
varieties which are not isomorphisms.

(3)\, The  blow-up of the center $C$ which coincides with the whole variety $X$ is an empty set.
The main feature which characterizes the notion of blow-up is the following ``restriction property'':
\begin{itemize}
\item[]If $X$ is a smooth variety containing a smooth subvariety $Y\subset X$, which contains the center $C\subset Y$
then the blow-up $\sigma_{C,Y}:\tilde{Y}\to Y$ at $C$ coincides with the strict transform of $Y$ under the blow-up $\sigma_{C,X}:\tilde{X}\to X$, i.e.,
$$\tilde{Y}\simeq \overline{\sigma_{C,X}^{-1}(Y\setminus C).}$$
\end{itemize}
\end{remarks}

{\bf Inductive setup.} Let $(X,{\cI},E,\mu)$ denote an arbitrary marked ideal. We will present 
an algorithm which establishes the following assertion, by induction on the dimension of  
$X$. 

\begin{theorem}\label{main}
There is an associated  resolution $(X_i)_{0\leq i\leq m_X}$, called \underline{canonical},
 satisfying the following conditions:
\begin{enumerate}
\item For any  surjective \'etale morphism $\phi: X '\to X$ the induced sequence
 $(X'_i)=\phi^*(X_i)$ is  the canonical resolution of $(X',{\cI}',E',\mu):=\phi^*(X,{\cI},E,\mu)$.
\item For any  \'etale  mophism $\phi: M'\to M$ the induced sequence
 $(X'_i)=\phi^*(X_{i})$ is an extension   of the canonical resolution of  $(X',{\cI}',E',\mu):=\phi^*(X,{\cI},E,\mu)$.
\end{enumerate}
\end{theorem}

\begin{proof} If $\cI= 0$ and $\mu>0$ then $\supp(X, {\cI},\mu)=X$, and the blow-up of $X$ is the empty set and thus it defines a unique resolution. Assume that $\cI\neq 0$.

We will use induction  on the dimension of $X$. If $X$ is $0$-dimensional, $\cI\neq 0$ and $\mu>0$ then $\supp(X, {\cI},\mu)=\emptyset$ and all resolutions are trivial.
\medskip

{\bf Step 1}.  {\bf Resolving a marked ideal $(X,{\cI},E,\mu)$ of maximal
order.}
Before performing the  resolution algorithm  for the marked ideal $({\cI},\mu)$ of maximal order in Step 1 we will replace it with the equivalent  homogenized ideal ${\cC}({\cH}({\cI},\mu))$.
Resolving the ideal ${\cC}({\cH}({\cI},\mu))$ defines a resolution of $({\cI},\mu)$ at this step.
To simplify notation we shall denote ${\cC}({\cH}({\cI},\mu))$ by 
$\cJ = (\cJ, \mu(\cJ))$.
\medskip

 {\bf Step 1a}. {\bf Reduction to the nonboundary case. Moving $\supp\,\cJ$ and $H^s_\alpha$ apart .}
 For any multiple  blow-up $(X_i)$ of $(X,\cJ,E,\mu(\cJ))$, we will denote (for simplicity) the strict transform of $E$ on any $X_i$ also by $E$.

For any $x\in X_i$, let $s(x)$ denote the number of divisors in $E$ through $x$ and
set $$ s_i=\max\{ s(x) \mid x \in \supp\,\cJ_i)\}.$$

Let $s=s_0$. By  assumption the intersections of any $s> s_0$ components of the exceptional divisors are disjoint from
$\supp\,\cJ$.
Each intersection of divisors in $E$ is locally defined by the intersection of some irreducible components of these divisors.
Find all   intersections $H^s_\alpha$, $\alpha\in A$, of $s$ irreducible components of divisors $E$ such that $\supp\,\cJ \cap H^s_\alpha\neq\emptyset$. By the maximality of $s$, the supports $\supp\,\cJ_{|H^s_\alpha})\subset H^s_\alpha$ are disjoint from $H^s_{\alpha'}$, where $\alpha'\neq\alpha$.

Set  $$H^s:=\bigcup_\alpha H^s_\alpha,\quad U^s:=X \setminus H^{s+1}, \quad \underline{H^s}:=H^s\setminus H^{s+1}.$$
Then $\underline{H^s}\subset U_s$ is a smooth closed subset $U_s$. Moreover $\underline{H^s}\cap \supp(\cJ)={H^s}\cap \supp(\cJ)$ is closed.

Construct the canonical resolution of $\cJ_{|\underline{H^s}}$. By Lemma \ref{le: S}, it
defines a multiple  blow-up of $(\cJ,\mu(\cJ))$ such that
$$\supp\,\cJ_{j_1}\cap{H^s_{j_1}}=\emptyset.$$
In particular the number of the strict tranforms of $E$ passing through a single point of the support drops $s_{j_1}<s$. Now we put $s=s_{j_1}$ and repeat the procedure. We continue the above process until $s_{j_k}=s_r=0$.
Then $(X_j)_{0\leq j\leq r}$  is a multiple  blow-up of $(X,\cJ,E,\mu(\cJ))$
such that $\supp\,\cJ_r$ does not intersect any divisor in $E$.

Therefore $(X_j)_{0\leq j\leq r}$ and further longer multiple  blow-ups $(X_j)_{0\leq j\leq m}$ for any $m\geq r$   can be considered as multiple  blow-ups of $(X,\cJ,\emptyset,\mu(\cJ))$ since starting from $X_r$ the strict transforms of $E$ play no further role in the resolution process since they do not intersect  $\supp\,\cJ_j$ for $j\geq r$. We reduce the situation to the ``nonboundary case" $E = \emptyset$.
\medskip

{\bf Step 1b}.   {\bf Nonboundary case.}
Let $(X_j)_{0\leq j\leq r}$ be the multiple  blow-up of $(X,\cJ,\emptyset,\mu(\cJ))$ defined in Step 1a.

For any $x\in \supp\,\cJ \subset X$, find a tangent direction 
$u_\alpha \in {\cD}^{\mu(\cJ)-1}(\cJ)$  on some neighborhood $U_\alpha$ of $x$.
Then $V(u_\alpha)\subset U_\alpha$ is a hypersurface of maximal contact.
By the quasicompactness of $X$, we can assume that the covering defined by $U_\alpha$ is finite.
Let $U_{i\alpha}\subset X_i$ be the inverse image of $U_{\alpha}$ and let $H_{i\alpha}:=V(u_\alpha)_i\subset U_{i\alpha}$ denote the strict transform of $H_\alpha:=V(u_\alpha)$.

Set (see also \cite{Kollar2}) $$\widetilde{X}:= \coprod U_\alpha\quad \quad  \widetilde{H}:=\coprod H_\alpha\subseteq \widetilde{X}.$$
The closed embeddings $H_\alpha \subseteq U_\alpha$ define the closed embedding $\widetilde{H}\subset \widetilde{X}$ of a hypersurface of maximal contact $\widetilde{H}$.

Consider the surjective \'etale morphism
$$\phi_U: \widetilde{X}:= \coprod U _\alpha\to X.$$
Denote by $\widetilde{\cJ}$ the pull back of the ideal sheaf $\cJ$ via $\phi_U$.
The multiple  blow-up  $(X_{i})_{0\leq i\leq r}$ of $\cJ$ defines a multiple  blow-up $(\widetilde{X}_{0\leq i\leq r})$ of $\widetilde{\cJ}$  and a multiple  blow-up $(\widetilde{H}_i)_{0\leq i\leq r}$ of $\widetilde{\cJ}_{|H}$.

Construct the canonical resolution of $(\widetilde{H}_{i})_{r\leq i\leq m}$ of the marked ideal ${\widetilde{\cJ}}_{r|\widetilde{H}_{r}}$ on $\widetilde{H}_r$. It defines, by Lemma \ref{le: coeff}, a resolution
$(\widetilde{X}_{r\leq i\leq m})$ of ${\widetilde{\cJ}}_{r}$ and thus also a resolution  $(\widetilde{X}_i)_{0\leq i\leq m}$
of $(\widetilde{X},{\widetilde{\cJ}},\emptyset, \mu(\cJ))$. Moreover both resolutions are related by the property $$\supp(\widetilde{\cJ}_{i})=\supp(\widetilde{\cJ}_{i|\widetilde{H}_i}).$$

Consider a  (possible) lifting  of $\phi_U$: $$\phi_{iU}: \widetilde{X}_i:= \coprod U _{i\alpha}\to X_i,$$ which is a  surjective locally \'etale morphism.
The lifting is constructed for $0\leq i\leq r$.

For ${r\leq i\leq m}$ the resolution $\widetilde{X}_i$ is induced by the canonical resolution $(\widetilde{H}_i)_{r\leq i\leq m}$ of $\cJ_{r|\widetilde{H}_r}$

We show that the resolution $(\widetilde{X_i})_{r\leq i\leq m}$
descends to the resolution $({X_i})_{r\leq i\leq m}$.

Let $\widetilde{C}_{j_0}=\coprod C_{j_0\alpha}$ be the center of the blow-up $\widetilde{\sigma}_{j_0}:\widetilde{X}_{j_0+1}\to\widetilde{X}_{j_0}$. The closed subset $C_{j_0\alpha}\subset U_{j_0\alpha}$ defines the center of an extension of the canonical resolution $(H_{j\alpha})_{r\leq j\leq m}$.

 If  ${C}_{j_0\alpha}\cap   U_{j_0\beta}\neq \emptyset $, then, by the canonicity and condition (2) of the inductive assumption, the subset
${C}_{j_0\alpha\beta}:={C}_{j_0\alpha}\cap   U_{j_0\beta}$ defines  the center
of an extension of  of the canonical resolution $H_{j\alpha\beta}:=((H_{j\alpha}\cap U_{j\beta}))_{r\leq j\leq m}$.
On the other hand ${C}_{j_0\beta\alpha}:={C}_{j_0\beta}\cap   U_{j\alpha}$ defines
 the center of an extension of  the canonical resolution $((H_{j\beta\alpha}:=H_{j\beta}\cap U_{j\alpha}))_{r\leq j\leq m}$.


By the Glueing Lemma \ref{le: homo} for the tangent directions $u_\alpha$ and $u_\beta$,
there exist \'etale neighborhoods  $\phi_{u_\alpha},\phi_{u_\beta}: \overline{U}_{\alpha\beta}\to {U}_{\alpha\beta}:=U_\alpha\cap U_\beta $ of $x=\phi_u(\overline{x})=\phi_v(\overline{x}) \in X$,  where $\overline{x}\in \overline{X}$, such that
\begin{enumerate}
\item  $\phi_{u_\alpha}^*({\cJ})=\phi_{u_\beta}^*(\cJ)$;
\item  $\phi_{u_\alpha}^*(E)=\phi_{u_\beta}^*(E)$;
\item  $\phi_{u_\alpha}^{-1}(H_{j\alpha\beta})=\phi_{u_\beta}^{-1}(H_{j\beta\alpha})$;
\item $\phi_{u_\alpha}(\bar{x})=\phi_{u_\beta}(\bar{x})$ for $\overline{x}\in \supp(\phi_{u_\alpha}^*({\cJ}))$.
\end{enumerate}
Moreover, all the properties lift to the relevant \'etale morphisms    $\phi_{u_{\alpha i}},\phi_{u_{\beta i}}: \overline{U}_{\alpha\beta i}\to {U}_{\alpha\beta i}$.
Consequently, by canonicity, $\phi_{u_\alpha j_0}^{-1}({C}_{j_0\alpha\beta})$ and $\phi_{u_\beta j_0}^{-1}(C_{j_0\beta\alpha})$  both define the next center of the extension of the canonical resolution
$\phi_{u_\alpha}^{-1}(H_{j_0\alpha\beta})=\phi_{u_\beta}^{-1}(H_{j_0\beta\alpha})$ of $\phi_{u_\alpha j_0}^*({\cJ}_{|H_{\alpha\beta}})=\phi_{u_\beta}^*({\cJ}_{|H_{\beta\alpha}})$.
Thus $$\phi_{u_\alpha}^{-1}({C}_{j_0\alpha\beta})=\phi_{u_\beta}^{-1}({C}_{j_0\beta\alpha}),$$ and finally, by property (4),
 $${C}_{j_0\alpha\beta}={C}_{j_0\beta\alpha}.$$
Consequently $\widetilde{C}_{j_0}$ descends to a smooth closed center $C_{j_0}=\bigcup {C}_{j_0\alpha} \subset X_{j_0}$ and
the resolution $(\widetilde{X_i})_{r\leq i\leq m}$
descends to a resolution $({X_i})_{r\leq i\leq m}$.
\medskip

{\bf Step 2}.   {\bf Resolving marked ideals $(X,{\cI},E,{\mu})$.}
For any  marked ideal $(X,{\cI},E,\mu)$ write
 $$I=\cM({\cI}){\cN}({\cI}),$$  \noindent where $\cM({\cI})$ is the {\it monomial
part} of ${\cI}$, that is, the product of the principal ideals
defining the irreducible components of the divisors in $E$, and ${\cN}({\cI})$ is the {\it nonmonomial
part} which is not divisible by any ideal of a divisor in $E$.
Let $$\ord_{{\cN}({\cI})}:=\max\{\ord_x({\cN}({\cI}))\mid x\in\supp(\cI,\mu)\}.$$

\begin{definition} (Hironaka, Bierstone-Milman,Villamayor, Encinas-Hauser)
By the {\it companion ideal} of $({\cI},\mu)$  where $I={\cN}({\cI})\cM({\cI})$
we mean the marked ideal of maximal order
\begin{displaymath}
O({\cI},\mu)= \left\{ \begin{array}{ll}
({\cN}({\cI}),\ord_{{\cN}({\cI})})\, + \,(\cM({\cI}),\mu-\ord_{{\cN}({\cI})})& \textrm{if \,  $\ord_{{\cN}({\cI})}<\mu$},\\ ({\cN}({\cI}),\ord_{{\cN}({\cI})})&\textrm{if \,  $\ord_{{\cN}({\cI})}\geq \mu$}.
\end{array} \right.
\end{displaymath}
\end{definition}

In particular $O({\cI},\mu)=({\cI},\mu)$ for ideals $({\cI},\mu)$ of maximal order.
\medskip

{\bf Step 2a}. {\bf Reduction to the monomial case by using companion ideals.}
By Step 1 we can resolve the marked ideal of maximal order 
$O({\cI})=O({\cI},\mu)=(O({\cI}),\mu(O({\cI})))$.  By Lemma \ref{le: operations}, for any multiple  blow-up of  $O({\cI},\mu)$,
\begin{align*}
\supp(O({\cI},\mu))_i&= \supp[{\cN}({\cI}),\ord_{{\cN}({\cI})}]_i\,\, \cap \,\, \supp[M({\cI}),\mu-\ord_{{\cN}({\cI})}]_i \\
&= \supp[{\cN}({\cI}),\ord_{{\cN}({\cI})} ]_i\,\,\cap \,\, \supp({\cI}_i,\mu).
\end{align*}
Consequently, such a resolution leads to an ideal $({\cI}_{r_1},\mu)$ such that $\ord_{{\cN}({\cI}_{r_1})}<\ord_{{\cN}({\cI})}$. Then we repeat the procedure for $({\cI}_{r_1},\mu)$.
 We find marked ideals $({\cI}_{r_0},\mu)=(\cI, \mu), ({\cI}_{r_1},\mu), \ldots, ({\cI}_{r_m},\mu)$ such that
$\ord_{{\cN}({\cI}_0)}>\ord_{{\cN}({\cI}_{r_1})}>\ldots>\ord_{{\cN}({\cI}_{r_m})}$.
  The procedure terminates after a finite number of steps when we arrive at an ideal $({\cI}_{r_m},\mu)$ with $\ord_{{\cN}({\cI}_{r_m})}=0$ or with $\supp({\cI}_{r_m},\mu)=\emptyset$. In the second case we get a resolution. In the first case $\cI_{r_m}=\cM({\cI}_{r_m})$ is monomial.
\medskip

{\bf Step 2b}. {\bf Monomial case $\cI=\cM({\cI})$}.
The collection of divisors $E$ is ordered (see Definitions 2.1.1 and 2.1.3); say
$E = \{D_1, D_2, \ldots\}$. Let $\Sub(E)$ denote the set of all subsets of $E$. The ordering of $E$ induces a natural lexicographic order on
$\Sub(E)$: We can associate to each $S \in \Sub(E)$ the lexicographic order of the binary
sequence $(\delta_1, \delta_2,\ldots)$, where $\delta_i = 0$ or $1$ according as $D_i \notin S$
or $D_i \in S$.  (The actual formula for the order is irrelevant as long as it is canonical and linear for the divisors passing through a point.)

Let $x_1,\ldots,x_k$ define equations of the components ${D}^x_1,\ldots,{D}^x_k\in E$ through $x\in \supp(X,{\cI},E,\mu)$ and let
 ${\cI}$  be generated by a monomial $x^{(a_1,\ldots,a_k)}$ at $x$. In particular $$\ord_x(\cI)(x):=a_1+\ldots+a_k.$$

 Let
 $\rho(x):=\{D_{i_1},\ldots,D_{i_l}\}\in \Sub(E)$  be the maximal (with respect to the order on $\Sub(E)$) subset satisfying the properties
\begin{enumerate}
\item  $a_{i_1}+\ldots+a_{i_l}\geq \mu.$
\item For any $j=1,\ldots,l$, $a_{i_1}+\ldots +\check{a}_{i_j}+\ldots +a_{i_l} < \mu. $
\end{enumerate}
Let $R(x)$ denote the subsets in $\Sub(E)$ satisfying the properties (1) and (2).
The maximal irreducible components of the $\supp({\cI},\mu)$ through $x$ are described by the intersections
$\bigcap_{D\in A} D$ where $A\in R(x)$. In particular $\supp({\cI},\mu)$ is a union of components with simple normal crossings.

The maximal locus of $\rho$ determines at most one maximal component of $\supp({\cI},\mu)$ through each $x$. The invariant $\rho$ is introduced to describe the center of the blow-up in a unique way. As we see below to resolve monomial case we can randomly pick any maximal irreducible component of  $\supp({\cI},\mu)$. The algorithm is controlled by the order and there is no need to introduce additional invariants unless we would like to construct invariants which describe the center and decrease after blow-up. (see \cite{BM2}, \cite{EV})

  After the blow-up at the maximal locus $C=\{x_{i_1}=\ldots=x_{i_l}=0\}$ of $\rho$, the ideal $\cI=(x^{(a_1,\ldots,a_k)})$ is equal to $\cI'=({x'}^{(a_1,\ldots,a_{i_j-1},a,a_{i_j+1},\ldots,a_k)})$ in the neighborhood  corresponding to $x_{i_j}$, where $a=a_{i_1}+\ldots+a_{i_l}-\mu<a_{i_j}$. In particular the invariant $\ord_x(\cI)$ drops for all  points of some maximal components of $\supp({\cI},\mu)$. Thus the maximal value of $\ord_x(\cI)$ on the maximal components of  $\supp({\cI},\mu)$ which were blown up is bigger than the maximal value of $\ord_x(\cI)$  on the new maximal components of $\supp({\cI},\mu)$.
The algorithm
terminates after a finite number of steps.
\end{proof}


\subsection{Summary of the resolution algorithm}

The resolution algorithm can be represented by the following scheme.

\noindent {\bf Step 2}.  Resolve $(\cI,\mu)$.
\begin{itemize}
\item[]
{\bf Step 2a.} Reduce $(\cI,\mu)$ to the monomial marked
 ideal $\cI=\cM(\cI)$.
$$\Downarrow$$
\item[]
If $\cI\neq\cM(\cI)$, decrease the maximal order of the nonmonomial part $\cN(\cI)$ by resolving the companion
ideal $O(\cI) = O(\cI,\mu)$.
\begin{itemize}
\item[]
{\bf Step 1}. Resolve the companion ideal $O(\cI)$:\\ 
Replace $O(\cI)$ with $\cJ:=\cC(\cH(O(\cI)))\simeq O(\cI)$. (*)
\medskip

\begin{itemize}
\item[]
{\bf Step 1a}. Move apart all strict transforms of $E$ and $\supp\,\cJ$.
$$\Downarrow$$
\item[]
Move away all intersections $H^s_\alpha$
of $s$ divisors in $E$
(where $s$ is the maximal number of divisors in $E$ through
points in $\supp\,\cJ = \supp\,O(\cI)$).
$$\Updownarrow$$
\item[]
For any $\alpha$,  resolve $\cJ_{ | (\bigcup_\alpha H^s_\alpha)}$.
\medskip

\item[]
{\bf Step 1b}. If  the strict transforms of $E$ do not intersect  $\supp\,\cJ$, resolve $\cJ$.
$$\Updownarrow$$
\item[]
Simultaneously resolve
 all $\cJ_{|V(u)}$ , where $V(u)$ is a hypersurface of
 maximal contact. (Use the property of homogenization (\cite{Wlod}), 
and Koll\'ar's trick (\cite{Kollar2}).
\end{itemize}
\end{itemize}
\medskip
\item[]
{\bf Step 2b}. Resolve the monomial marked ideal $\cI=\cM(\cI)$.
\end{itemize}
\medskip

\begin{remarks}
(1)\, (*) The ideal $O(\cI)$ is replaced with $\cH(O(\cI))$ to ensure that 
the algorithm constructed in Step 1b is independent of the 
choice of the tangent direction $u$.

We replace $\cH(O(\cI))$  with $\cC(\cH(O(\cI)))$ to ensure the equalities 
$\supp(\cJ_{|S})=\supp(\cJ)\cap S$, where $S=H^s_{\alpha}$ in Step 1a and  
$S=V(u)$ in Step 1b.

(2)\, If $\mu=1$ the companion ideal is equal to $O(\cI,1)=(\cN(\cI),\mu_{\cN(\cI)})$ so the general strategy of the resolution of $(\cI,\mu)$ is to decrease the order of the nonmonomial part and then to resolve the monomial part.

(3)\, In particular, in order to desingularize $Y$, we put  $\mu=1$ and 
$\cI=\cI_Y$, where $\cI_Y$ is the sheaf of the subvariety $Y$,
 and we resolve the marked
ideal $(X,\cI,\emptyset,\mu)$. The nonmonomial part $\cN(\cI_i)$ is nothing but the \underline{weak transform} $(\sigma^i)^{\rm w}(\cI)$ of $\cI$.
\end{remarks}

In the next sections, we provide
a complexity bound for the algorithm.

\section{Complexity bounds on a blow-up}

Our purpose for the rest of the paper is to estimate the complexity of the  desingularization algorithm
described in the previous sections.

\subsection{Preliminary setup}

\subsubsection{Affine marked ideals}  \label{4.1.1}
An input of the algorithm is an {\it affine marked ideal}; that is, a collection of tuples
$$\cT:=(\{X_{\alpha,\beta},\cI_{\alpha,\beta},E_{\alpha,\beta},U_{\alpha,\beta},(\CC^{n_\alpha})_\alpha\mid \alpha \in A, \beta \in B_\alpha\},\mu),$$
  where:
\begin{enumerate}
\item $(\CC^{n_\alpha})_{\alpha}\simeq \CC^{n_\alpha}$.
\item   $\{U_{\alpha,\beta}\mid \beta\in B_\alpha\}$ is an open cover of $(\CC^{n_\alpha})_{\alpha}$.

\item $U_{\alpha,\beta} \subset (\CC^{n_\alpha})_\alpha$ is an open subset whose complement is given by $f_{\alpha,\beta} =0$.

\item $X_{\alpha,\beta}\subset (\cCC^{n_\alpha})_\alpha$ is a closed subset such that $X_{\alpha,\beta}\cap U_{\alpha,\beta}$ is
 a nonsingular $m$-dimensional variety (possibly reducible). Moreover there exists  a set of parameters (coordinates) on $U_{\alpha,\beta}$,
$$u_{{\alpha,\beta},1},\ldots,u_{{\alpha,\beta},{n_\alpha}} \in  \CC[x_{\alpha,1},
\dots,x_{\alpha,{n_\alpha}}], $$
such that  $u_{{\alpha,\beta,i}}$ is a coordinate $x_{\alpha,j}$ describing an exceptional divisor or it is transversal
to the exceptional divisors (over $U_{\alpha,\beta}$) , and, moreover, $\cI_{X_{\alpha,\beta}}=(u_{{\alpha,\beta},i_1},\ldots,u_{{\alpha,\beta},i_k} ) \subset  \cCC[x_{\alpha,1}, . . . , x_{\alpha,{n_\alpha}}]$, for a certain subset $\{i_1,\ldots,i_k\}\subset \{1,\ldots,{n_\alpha}\}$.

\item $\cI_{\alpha,\beta}=\langle g_{{\alpha,\beta},1},\dots,g_{{\alpha,\beta},\overline j}\rangle\subset \cCC[x_{\alpha,1},
\dots,x_{\alpha,{n_\alpha}}]$ is an ideal,

\item $E_{\alpha,\beta}$ is a collection of $s\leq m$ smooth divisors in $(\cCC^{n_\alpha})_\alpha$ described by some $x_{\alpha,{j}}=0$, where $j=n_\alpha-s+1,\ldots,n_\alpha$. The divisors in  $E_{\alpha,\beta}$ are either transversal to $X_{\alpha,\beta}$ or they contain $X_{\alpha,\beta}$.
The restrictions of the divisors in $E_{\alpha,\beta}$  that are transversal to $X_{\alpha,\beta}$ define the exceptional divisors on $X_{\alpha,\beta}$ .

\item There exist birational maps $i_{\alpha_1\beta_1,\alpha_2,\beta_2}: X_{\alpha_1,\beta_1}\dashrightarrow X_{\alpha_2,\beta_2}$ given by
$$X_{\alpha_1,\beta_1} \ni  x\mapsto i_{\alpha_1\beta_1,\alpha_2,\beta_2}(x)=(\frac{v_{\alpha_1\beta_1,\alpha_2,\beta_2,1}}{w_{\alpha_1\beta_1,\alpha_2,\beta_2,1}},\ldots, \frac{v_{\alpha_1\beta_1,\alpha_2,\beta_2,n_{\alpha_2}}}{w_{\alpha_1\beta_1,\alpha_2,\beta_2,n_{\alpha_2}}})(x) \in X_{\alpha_2,\beta_2}$$
for regular functions $v_{\alpha_1\beta_1,\alpha_2,\beta_2,1},\ldots, v_{\alpha_1\beta_1,\alpha_2,\beta_2,n_{\alpha_2}},w_{\alpha_1\beta_1,\alpha_2,\beta_2,1},\ldots, w_{\alpha_1\beta_1,\alpha_2,\beta_2,n_{\alpha_2}}$ in \\ $\CC[x_{\alpha_1,1},\ldots,x_{\alpha_1, n_{\alpha_1}}]$.

\item The birational maps $i_{\alpha\beta,\alpha',\beta'}$ determine uniquely up to an isomorphism a variety $X_\cT$ in the following sense: There exist open embeddings
$j_{\alpha,\beta}:X_{\alpha,\beta}\cap U_{\alpha,\beta}\hookrightarrow X_\cT$ defining an open cover of
   $X_\cT$, and satisfying
   $$j^{-1}_{\alpha_2,\beta_2}j_{\alpha_1,\beta_1}=i_{\alpha_1\beta_1,\alpha_2,\beta_2}\,.$$
\item $\mu\ge 0$
is an integer.

\item $\supp(\cI_{\alpha,\beta},\mu)\cap U_{\alpha,\beta}\cap U_{{\alpha,\beta}'}=\supp(\cI_{{\alpha,\beta}'},\mu)\cap U_{\alpha,\beta}\cap U_{{\alpha,\beta}'}$.
\end{enumerate}
\medskip

\begin{remarks} 
(1)\, The objects $X_{\alpha,\beta}, E_{\alpha,\beta}, \cI_{\alpha,\beta}$, as well as corresponding functions $g_{\alpha,\beta,i},u_{\alpha,\beta,j}, x_{\alpha,i}$  are relevant for the algorithm after they are restricted  to $U_{\alpha,\beta}$. Their behavior in the complement $(\CC^{n_\alpha})_\alpha \setminus U_{\alpha,\beta}$ has no relevance.

(2)\, The operation of restricting to maximal contact  leads to considering open subsets $U_{\alpha,\beta} \subset (\CC^{n_\alpha})_\alpha$.

(3)\, While studying the complexity of the algorithm we will  assume that the coefficients of  the input polynomials belong just  to  $\ZZ$. All general considerations are given  for coefficients from  $\CC$, and remain valid for any algebraically closed field of zero characteristic.

(4)\, The open embeddings $j_{\alpha,\beta}$ can be constructed from $i_{\alpha_1\beta_1,\alpha_2,\beta_2}$ after performing the algorithm but we do not dwell 
on this.
\end{remarks}
\medskip

\begin{definition} By the support of $\cT$ we mean the collection of the sets
$$\supp(\cT):=\{\supp(X_{\alpha,\beta}\cap U_{\alpha,\beta},\cI_{\alpha,\beta},E_{\alpha,\beta}\cap U_{\alpha,\beta}\cap X_{\alpha,\beta},\mu)\}_{{\alpha\in A,\beta\in B}}$$
\end{definition}
\medskip

\begin{definition} Given an affine marked ideal
$\cT:=(\{X_{\alpha,\beta},\cI_{\alpha,\beta},E_{\alpha,\beta},U_{\alpha,\beta},(\CC^{n_\alpha})_\alpha\mid \alpha \in A, \beta \in B_\alpha\},\mu),$ we say that
an affine marked ideal $\cT':=(\{X_{\alpha',\beta'},\cI_{\alpha,\beta},E_{\alpha,\beta},U_{\alpha,\beta},(\CC^{n_\alpha})_\alpha\mid \alpha \in A', \beta \in B'_\alpha\},\mu),$ {\it  is defined over $\cT$, }provided:
\begin{enumerate}
\item There exist maps of index-sets  $p:A'\to A$ , and $p_{\alpha'}: B'_{\alpha'} \to B_{p(\alpha')}$.
\item The canonical projection on the first $\alpha=p(\alpha')$ components $$\pi_\alpha': (\CC^{n_{\alpha'}})_{\alpha'}\to  (\CC^{n_\alpha})_\alpha$$ \noindent determine
birational morphisms
$\pi_{\alpha',\beta'}:=\pi'_{|X_{\alpha',\beta'}}: X_{\alpha',\beta'} \to X_{\alpha,\beta}$
commuting with $i_{\alpha_1\beta_1,\alpha_2,\beta_2}$ and $i_{\alpha'_1\beta'_1,\alpha'_2,\beta'_2}$.
\item  There exist natural birational morphisms $X_{\cT'}\to X_{\cT}$ commuting with $j_{\alpha,\beta}$, and $j_{\alpha',\beta'}$.
\end{enumerate}
\end{definition}
\medskip

We introduce the following functions to characterize the affine marked ideal $$\cT:=(\{X_{\alpha,\beta},\cI_{\alpha,\beta},E_{\alpha,\beta},U_{\alpha,\beta}\}_{\alpha,\beta},(\CC^{n_\alpha})_\alpha ,\mu):$$

\begin{enumerate}
\item $m(\cT)=\dim(X_\cT),$
\item $\mu(\cT)=\mu,$
\item $d(\cT)$ is the maximal degree of all polynomials in
$$
\Psi(\cT):=\{u_{{\alpha,\beta},i}, g_{{\alpha,\beta},i}, f_{{\alpha,\beta}},  v_{\alpha,\beta,\alpha_2,\beta_2,i}, w_{\alpha,\beta,\alpha_2,\beta_2,j}.\}
$$
\item  $n(\cT)=\max{n_\alpha}$,
\item  $l(\cT)$ is the maximal number of all polynomials in $\Psi(\cT)$.

\item $q(\cT)$ is the number of neighborhoods $U_{\alpha,\beta}$ in $\cT$ ( i. e the number of  the indices such that  $\alpha\in A, \beta \in B$).
\item  $b(\cT)$  the maximum bit size of any (integer) coefficient of each of the polynomials  in $\Psi(\cT)$
 \end{enumerate}
\medskip

\begin{remark} The function $b(\cT)$ is used only for the estimation of the  total complexity of the algorithm.
In particular it has no relevance for the estimates of the number of
blow-ups,  the maximal embedding dimension $n(\cT)$, or the number
of neighborhoods.
\end{remark}

Algorithmically the input is represented by the coefficients of polynomials describing an affine marked ideal $\cT_0$.
We assume $$m(\cT_0)=m,\quad  d(\cT_0)\leq d_0, \quad  n(\cT_0)\leq n_0, \quad  l(\cT_0)\leq l_0,  \quad  b(\cT_0)\leq b_0.$$
Then in particular, the
total bit-size of the input does not exceed $b_0\cdot l_0 \cdot d_0^{O(n_0)}$, cf. \cite{G2}.
\medskip

\subsubsection{Resolution of singularities}
For simplicity consider an irreducible affine variety $Y\subset \CC^n$ described by some equations. The algorithm resolves $Y$ by the following procedure:
\medskip

{\bf Step A}. Find the generators of $\cI_Y=\langle g_1,\dots,g_{\overline j}\rangle\subset \cCC[x_1,
\dots,x_n]$ and construct the affine marked ideal $$\cT:=(X=\CC^n, \cI_Y, E=\emptyset, U=\CC^n,\CC^n,\mu=1)$$

{\bf Step B}. Start the resolving procedure for  the affine marked ideal $(\{X=\CC^n, \cI_Y, E=\emptyset, U=\CC^n\},\CC^n, \mu=1)$ (see below).
\medskip

{\bf Step C}. Pick a nonsingular point $p \in Y \subset \CC^n$. Stop the resolution procedure when the constructed center of the following blow-up in the algorithm passes through the inverse image of $p$. As an output of the resolution algorithm we  get  an affine marked ideal
$$\cT':=(\{ X_{\alpha',\beta'},\cI_{\alpha',\beta'},E_{\alpha',\beta'},U_{\alpha,\beta}, \CC^{n}_{\alpha'}\}_{\alpha',\beta'},1)$$
\noindent over $\cT$.  In particular we have a collection of projections $$\pi_{\alpha'}: (\CC^{{n}_{\alpha'}})_{\alpha'}\to \CC^{n}$$ \noindent (projection on the first $n$  coordinates.).
(Note that the restriction of $\pi_\alpha$ defines a  birational morphism $\pi_{\alpha',\beta'}: X_{\alpha',\beta'} \to X$ which is an isomorphism in a neighborhood of $p\in X$.)

The center of the following blow-up is described on  some open subcover $\{ U_{{\alpha',\beta''}}\}_{\alpha'\in A', \beta''\in B''_{\alpha'}} $  of \\
$\{U_{\alpha',\beta'}\}_{\alpha'\in A', \beta'\in B'_{\alpha'}}$
by  $C_{{\alpha',\beta''}}\cap U_{{\alpha',\beta''}}$,
for closed subsets
$C_{{\alpha',\beta''}}\subset (\CC^{n_{\alpha'}})_{\alpha'}$.
Consider the unique irreducible  component $\tilde{C}_{{\alpha',\beta''}}$ of $C_{{\alpha',\beta''}}$ containing the inverse image of  the point $p$.
Then $$\overline{\pi}_{\alpha',\beta''}:=\pi_{| \tilde{C}_{\alpha',\beta''}\cap  U_{{\alpha',\beta''}}}: \tilde{Y}_{\alpha',\beta''}:=\tilde{C}_{{\alpha',\beta''}}\cap  U_{{\alpha',\beta''}} \to Y$$ is a local resolution of $Y$.
The resolution space $\tilde{Y}$ is described by an open cover $ \{\tilde{Y}_{{\alpha',\beta''}}\,,\overline{\pi}_{\alpha',\beta''}\}_{{\alpha',\beta''}}$. The sets $\tilde{Y}_{{\alpha',\beta''}}$ are represented as closed subsets of the open subsets $U_{{\alpha',\beta''}}\subset \CC^{n_{\alpha'}}$.
\medskip

\subsubsection{Principalization}
Given  a smooth affine variety $X\subset \CC^n$, described by equations $u_{X,i}=0$, and an ideal $\cI=(g_1,\ldots,g_k)$ on $X$ and on $\CC^n$.
\medskip

{\bf Step A}. First the algorithm finds affine neighbourhoods $U_{\alpha,\beta}$ in $\cCC^n$ each given by an
inequality ${f_{\alpha,\beta}}\neq 0$ in which $X$ is represented by a family
of local parameters $$u_{1}=\cdots=u_{{n-m}}=0.$$
Moreover, it finds coordinates $x_i$ on $\CC^n$, such that (up to index permutation)
$$ u_{1},\cdots,u_{{n-m}},x_{n-m+1},\ldots, x_n $$
is a complete set of parameters.

The  local parameters $u_{i}$ are chosen among the input
polynomials. To this end the algorithm can for each choice of $\{i_1,\dots,i_{n-m}\}
\subset \{1,\dots,\overline i \}$ pick an identically non-vanishing minor
$f_{\alpha,\beta}$ of the Jacobian matrix of $\{u_i\}_i$.

We construct an affine marked ideal given by an
input tuple by $$\cT:=(\{X_{\beta}:=X, \cI_{\beta}:=\cI, E_{\beta}=\emptyset, U_{\beta},\CC^{n}\}_{\beta},\mu=1).$$
\smallskip

{\bf Step B}. The algorithm  resolves $\cT=(\{X_{\beta}:=X, \cI_{\beta}:=\cI, E_{\beta}=\emptyset, U_{\beta}, \CC^{n}\}_{\beta},\mu=1)$.
As an output, we get $$\cT':=(\{ X_{{\alpha',\beta'}},\cI_{{\alpha',\beta'}},E_{\alpha',\beta'},U_{\alpha',\beta'},(\CC^{n_{\alpha'}})_{\alpha'} \}_{{\alpha',\beta'},},1),$$
\noindent over $\cT$.
\medskip

{\bf Step C}. The variety $X':=X_{\cT'}$ is described by an open cover  $\{X_{\alpha',\beta'}\cap U_{{\alpha',\beta'}}\}_{\alpha',\beta'}$ for closed subsets  $X_{\alpha',\beta'}\subset  \CC^{n_{\alpha'}}$,
and  open subsets $U_{{\alpha',\beta'}} \subset \CC^{n_{\alpha'}}$ (see (8) from \ref{4.1.1}). Moreover,
we have a collection of birational morphisms  $\pi_{{\alpha',\beta}'}: X_{{\alpha',\beta}'}\cap U_{{\alpha',\beta'}} \to X\subset
\CC^n$.
The principal ideal on $X_{{\alpha',\beta'}}$ is generated by $$(g_1\circ \pi_{{\alpha',\beta}'}\,\ldots,g_k\circ \pi_{{\alpha',\beta}'}) = (x_{1}^{a_1}\ldots x_{n_{\alpha'}}^{a_{n_{\alpha'}}})$$
\smallskip

\subsection{\bf Description of blow-up}
Consider an affine marked ideal $$\cT:=(\{X_{\alpha,\beta},\cI_{\alpha,\beta},E_{\alpha,\beta},U_{\alpha,\beta},(\CC^{n_\alpha})_\alpha \mid \alpha \in A, \beta \in B_\alpha \},\mu)$$ corresponding to a marked ideal $(X,\cI,E,\mu)$.
Let  $C\subset X$ be a smooth center described  as follows:

 We assume that there is an open subcover cover $\{U_{{\alpha,\beta'}}\}_{\alpha\in A, \beta'\in B'_\alpha} \subset (\CC^{n_\alpha})_\alpha\simeq \CC^n$ of $U_{\alpha,\beta}$, together with a map of indices $\rho: B'_\alpha\to B_\alpha$, and a
collection of closed subvarieties $C_{{\alpha,\beta'}}\subset (\CC^{n_\alpha})_\alpha$  (of dimension $k_{\alpha\beta'}\leq m$), such that

\begin{enumerate}
\item $\bigcup_{\rho(\beta')=\beta} U_{{\alpha,\beta'}}= U_{{\alpha,\beta}}$;
\item $C_{{\alpha,\beta'}}\cap U_{{\alpha,\beta'}}\subset  \supp(\cI_{\alpha,\beta',\mu)}\cap U_{\alpha,\beta'}$;
\item $C_{{\alpha,\beta'}}$ is described on  each $U_{{\alpha,\beta'}}$ by a set of local parameters $$u_{{\alpha,\beta'},1},\ldots,u_{{\alpha,\beta'},{n_\alpha}-m},u_{{\alpha\beta'},{n_\alpha}-m+1},\ldots,u_{{\alpha,\beta'},{n_\alpha}-k_{\alpha\beta'}} \in \CC[x_1,\ldots,x_{n_\alpha}],$$
i.e.,
$$u_{{\alpha,\beta'},1}=\ldots=u_{{\alpha,\beta'},{n_\alpha}-m}=u_{{\alpha\beta'},{n_\alpha}-m+1}=\ldots=u_{{\alpha,\beta'},{n_\alpha}-k_{\alpha\beta'}}=0, $$
where $X_{\alpha,\beta}$ is described on $U_{\alpha,\beta}\supset U_{\alpha,\beta'}$ by  $u_{{\alpha,\beta'},1}=\ldots=u_{{\alpha,\beta'},{n_\alpha}-m}=0$;
\item $u_{{\alpha,\beta'},1},\ldots,u_{{\alpha,\beta'},{n_\alpha}-k_{\alpha\beta'}}$ are transversal to the exceptional divisors   (over $U_{\alpha,\beta'}$),
or coincide with coordinate functions describing the exceptional divisors.
\end{enumerate}

Denote by $$\cT':=(\{X_{\alpha',\beta'},\cI_{\alpha',\beta'},E_{\alpha',\beta'},U_{\alpha',\beta'},(\CC^{n_{\alpha'}})_{\alpha'} \mid \alpha' \in A', \beta' \in B'_{\alpha'} \},\mu)$$ \noindent  the resulting affine marked ideal
obtained from $\cT$ by the blow-up with the center $C$. Below we describe more precisely the ingredients of $\cT'$.
\medskip

{\it The open cover after blow-up.}
The blow-up creates a new collection of ambient affine spaces $(\CC^{n_{\alpha'}})_{\alpha'}$.
Namely, we can associate with  functions $u_{{\alpha,\beta'},i}$ on  $(\CC^{n_\alpha})_\alpha$,  where $i=1,\ldots, {n_\alpha}-k_{\alpha\beta'}$, the   ${n_\alpha}-k_{\alpha\beta'}$ affine charts
$$(\CC^{n_{\alpha'}})_{\alpha'}, \quad  \mbox{where}\quad  \alpha':=(\alpha,i),\quad i=1,\ldots, {n_\alpha}-k_{\alpha\beta'},\quad {n_{\alpha'}}:=2{n_\alpha}-k_{\alpha\beta'}$$ We  also create a new collection of open subsets $U_{\alpha',\beta'}\subset (\CC^{n_{\alpha'}})_{\alpha'}$ by taking the inverse images of $U_{\alpha,\beta'}\subset (\CC^{n_\alpha})_\alpha$ under the morphisms $(\CC^{n_{\alpha'}})_{\alpha'}\to (\CC^{n_\alpha})_\alpha$.
\medskip

{\it The birational maps.}
The natural projection $\pi_\alpha: (\CC^{n_{\alpha'}})_{\alpha'}\to (\CC^{n_\alpha})_\alpha$
on the first $n_\alpha$ components defines the birational morphism $\pi_{\alpha',\beta'}=\pi_{\alpha|X_{\alpha',\beta'}} : X_{\alpha',\beta'} \to X_{\alpha,\beta'}$, for any $\alpha,\beta'$, such that  $X_{\alpha,\beta'}\neq \emptyset$. This defines birational morphisms
$$i_{\alpha'_1\beta'_1,\alpha'_2,\beta'_2}:  X_{\alpha'_1,\beta'_1} \buildrel \pi_{\alpha'_1,\beta'_1} \over \to X_{\alpha_1,\beta'_1} \buildrel  i_{\alpha_1\beta'_1,\alpha'_2,\beta_2}  \over \dashrightarrow X_{\alpha_2,\beta'_2}\buildrel \pi^{-1}_{\alpha'_2,\beta'_2} \over   \leftarrow X_{\alpha'_2,\beta'_2}.$$

 Consider a blow-up $X_{\cT'}$ of $C\subset X_\cT$.  There exist open embeddings
$j_{\alpha',\beta'}:X_{\alpha',\beta'}\cap U_{\alpha',\beta'}\hookrightarrow X_{\cT'}$
induced by $j_{\alpha,\beta}:X_{\alpha,\beta'}\cap U_{\alpha,\beta'}\hookrightarrow X_{\cT}$,
defining an open cover of
   $X_{\cT'}$, and satisfying
   $$j^{-1}_{\alpha_2',\beta_2'} \circ j_{\alpha_1',\beta_1'}=i_{\alpha'_1\beta'_1,\alpha'_2,\beta'_2}.$$
\medskip

{\it Equations of blow-up.}
Without loss of generality  the blow-up in each of the ${n_\alpha}-k_{\alpha\beta'}$ affine charts $(\CC^{n_{\alpha'}})_{\alpha'}$,
 where $\alpha':=(\alpha,i),\quad i=1,\ldots, {n_\alpha}-k_{\alpha\beta'}$,  can be described as follows:
(For simplicity, we drop the  $\alpha,\beta$ indices below.)

Assume that the function $u_{i_0}$, $i_0\leq n-k$, defines the chart of the blow-up.
The blow-up of  $\CC^n$ is a closed subset $bl(\CC^n)$ of $\CC^{2n-k}$ described by the following equations
\begin{align*}
u_j-u_{i_0}x_{j+n}&=0\quad \mbox{ for} \quad 0< j \leq n-k, j\neq i_0,\\
u_{i_0}-x_{i_0+n}&=0.
\end{align*}
\medskip

{\it The exceptional divisors.}
The exceptional divisor for this blow-up is given by $u_{i_0}=0$ on $bl(\CC^n)$. Since $u_{i_0}=x_{{i_0+n}}$ we may represent it  by the coordinate
$x_{{i_0+n}}$ on $\CC^{2n-k}$.

The previous exceptional divisors keep their form $x_j=0$ if they do not describe $C$, or they convert to $x_{j+n} (=u_{j}/u_{i_0})$ if they were described by  the function $u_j\equiv x_j$.
\medskip

{\it The strict transform of $X$.}
Recall that $X$ is described by  $u_1=\ldots=u_{n-m}=0$ on $U\subset \CC^n$.
  The blow-up of $X=X_{\alpha,\beta}$ is a closed subset $X'\subset \CC^{2n-k}$ which is described by a new set of equations:

\begin{enumerate}
\item $u_j-u_{i_0}x_{j+n}=0 \quad \mbox{for} \quad 0< j \leq n-k,\quad   j\neq i_0$;
 \item $u_{i_0}-x_{i_0+n}=0$;
\item $ x_{j+n}=0 \quad
 \mbox{for}\quad
0< j \leq n-m, \quad  j\neq i_0$.
\end{enumerate}
In some situations we consider additionally the induced equation
\begin{enumerate}
\item[(4)] $ 1=0 \quad \mbox{if}\quad    0< i_0 \leq n-m $.
\end{enumerate}
(Note that the equations of the first two types describe the blow-up $bl(\CC^n)$ of $\CC^n$. The third and the fourth types of the equations $ x_{j+n}=u_j/u_{i_0}=0$, $j\neq i_0$
(or $1=u_{i_0}/u_{i_0}=0$ )  describe the strict transform of $X$ inside $bl(\CC^n)$. In the latter case {if}   $0< j=i_0 \leq n-m$ the strict transform is an empty set in the relevant chart. Still we shall keep the uniform description of the objects and their transformations, and do not eliminate any equations in the description of the empty set.)

\subsubsection{The generators of $\cI_{\alpha,\beta}$ after blow-up}

We will not compute the controlled transforms of the generators of  $\cI$ ($=\cI_{\alpha,\beta})$   (over $U$) directly. Instead we modify them first.
The generators $g_i$ of $\cI$ satisfy, by Lemma \ref{bb}, the condition $$g_i\cdot f^{r_i}\in \cI_C^{\mu}+\cI_X,$$
where $V(f)=\CC^n\setminus U\subset \CC^n$ (we have dropped the indices $\alpha,\beta$ here).

For any generator $g_i$ write
$$g_i\cdot f^{r_i}=\sum_{ a_{n-m+1}+\ldots+a_{n-k}=\mu} h_{(a_{n-m+1},\ldots, a_{n-k}),i} u_{n-m+1}^{a_{n-m+1}}\cdot\ldots\cdot u_{n-k}^{a_{n-k}}+ \sum_{j:=1\ldots, m-n} h_{ij} u_j\,.$$
Set $\bar{a}:=(a_{n-m+1},\ldots, a_{n-k}), {\bar u}^{\bar{a}}:=u_{n-m+1}^{a_{n-m+1}}\cdot\ldots\cdot u_{n-k}^{a_{n-k}}$. Then we can rewrite the above as

$$g_i\cdot f^{r_i}=\sum_{ |\bar{a}|=\mu} h_{\bar{a}}\bar{u}^{\bar{a}}+ \sum_{j:=1\ldots, m-n} h_{ij} u_j\,.$$

\noindent To bound $r_i$ and $deg(h_{\bar{a},i}),\,
deg(h_{ij})$ we first consider a similar equality,
$$g_i\cdot f^{R_i}=\sum_{ |\bar{a}|=\mu} H_{\bar{a}}\bar{u}^{\bar{a}}+ \sum_{j:=1\ldots, m-n} H_{ij} u_j\,,$$
for certain $R_i,\, H_{\bar{a},i},\, H_{ij}$. We introduce a new variable $z$ and we get
an equality
$$g_i=z^{R_i}\cdot (\sum_{ |\bar{a}|=\mu} H_{\bar{a}}\bar{u}^{\bar{a}}+ \sum_{j:=1\ldots, m-n} H_{ij} u_j)\, +\, g_i\cdot
(\sum_{0\le j\le r-1} (z\cdot f)^j)\cdot (1-z\cdot f)\,;$$
in other words, $g_i$ belongs to the ideal
generated by $\{\bar{u}^{\bar{a}}\},\,\{u_j\},\,1-z\cdot f$. Therefore
one can represent $g_i$ as
$$g_i= \sum_{{ |\bar{a}|=\mu}}\widetilde{H}_{\bar{a},i}\cdot \bar{u}^{\bar{a}}+ \sum_{j:=1\ldots, m-n} \widetilde{H}_{ij}\cdot u_j\,
+\, \widetilde{H}\cdot (1-z\cdot f)\,,$$
for suitable polynomials $\widetilde{H}_{\bar{a},i},
\,\widetilde{H}_{ij},\,\widetilde{H}$ with degrees less than $(d\cdot \mu)^{2^{O(n)}}$ due to \cite{Seidenberg74}, \cite{Giusti}, \cite{Mora}. Hence substituting in the latter equality $z=1/f$ and cleaning the
denominator we obtain the bound $(d\cdot \mu)^{2^{O(n)}}$ on $r_i,\, deg(h_{\bar{a},i}),\,
deg(h_{ij})$.
\medskip

{\it The generators after blow-up and their degree}.
Using the above we can describe the controlled transform of $\cI$ in terms of the controlled transforms of its modified generators.
Define  the modified generators of $\cI$ to be $$\bar{g}_i=\sum_{|\bar{a}|=\mu} h_{\bar{a}}\bar{u}^{\bar{a}}.$$

Then their controlled transforms are given by $$\sigma^c(\bar{g}_i)=u^{-\mu}_{n-k}\sigma^*(h_{\bar{a},i}u^{\bar{a}}). $$

Denote by $G(d,n,\mu)$ the bound on the degree of the resulting marked ideal $\cT'$ after a blow up applied to
a marked ideal $\cT$, provided that $d(\cT)\leq d$,  $n(\cT)\leq n$. Thus, by the above:

\begin{lemma}
$G(n,d,\mu)< (d\cdot \mu)^{2^{O(n)}}$.
\end{lemma}

\subsection{Elementary operations and elementary auxillary functions}
To estimate the complexity of the desingularization algorithm we introduce a few auxiliary
functions related to the ingredients of $\cT$. It is convenient to associate to $\cT$ with data 
$(m,d,n,l,q,\mu)$, the vector $$\gamma:=(r,m,d,n,l,q,\mu) \in {\bf Z}^7_{\geq 0},$$
where $r$ is the subscript of the element of the resolution  $(\cT_r)_{r=0,1,\ldots,}$

\subsubsection{ The effect of a single blow-up}    Summarizing the above we get
the following:

\begin{lemma}\label{boundG}  Consider the object $\cT:=(\{X_{\alpha,\beta},\cI_{\alpha,\beta},E_{\alpha,\beta},U_{\alpha,\beta}\}_{\alpha,\beta},n,\mu)$  with data $(m,\mu,d,n,l,q)$.
Let $\cT':=(\{X'_{\alpha',\beta'},\cI'_{\alpha',\beta'},E'_{\alpha',\beta'},U'_{\alpha',\beta'}\}_{\alpha',\beta'},{n_{\alpha'}},\mu)$ denote the object with data $(m,\mu,d',n',l',q')$ obtained from $\cT$ by a single blow-up at the center $C$ represented by the collection of closed sets $\{ C_{\alpha\beta''} \subset (\CC^{n_\alpha})_\alpha\} $ describing the center in  open subsets $U_{\alpha\beta''}\subset U_{\alpha\beta}$.
Assume that the maximal degree of the polynomials describing the center is less than $d$. Assume that $ q$ gives also a bound  for the number of open neighborhoods $U_{\alpha\beta''}$. Then
\begin{enumerate}
\item $d'\leq G(n,d,\mu)< (d\cdot \mu)^{2^{O(n)}}$;
\item  $n'\leq 2n$;
\item $l'\leq l+n$;
\item $q'\leq n\cdot q$;
\end{enumerate}
\end{lemma}

The effect of the single blow-up can be measured by the function
$$Bl(r,m,d,n,l,q,\mu):=(r+1,m,G(n,d,\mu),2n,l+n,n\cdot q,\mu).$$
The multiple effect of the $t$  blow-ups can be measured by the recursive function
$$\overline{Bl}(r,m,d,n,l,q,\mu,t)=Bl\circ \overline{Bl}(r,m,d,n,l,q,\mu,t-1),$$
or, more briefly, $$\overline{Bl}(\gamma,t)= Bl(\overline{Bl}(\gamma,t-1)).$$
Note that $$\overline{Bl}(r,m,d,n,l,q,\mu,t)=(r+t,m,\overline{G}(n,d,\mu,t),2^tn, l+2^{t-1}n,(2^{t+1}-1)\cdot n^tq,\mu)$$
for the relevant function $\overline{G}(n,d,\mu,t)$.

\section{Bounds on multiplicities and  degrees of  coefficient ideals}

\subsubsection{ The maximal multiplicity of $\cI$ on the subvariety $X\subset \CC^n$.}

Let $d$ be the maximal degree of $(X,\cI)$  on $\CC^n$.
Denote by $M(d,n)$ a bound on the multiplicity of the ideal $\cI$ on $X$. To estimate $M(d,n)$, we can
assume (after a linear transformation of the coordinates) that the order of the polynomial
$u_{i_j}-x_j$ is at least 2 for all $1\le j\le n-m$. For any polynomial $g\in \cI_{\alpha,\beta}$ one can
find polynomials $h_j \in \CC[x_1,\dots,x_n], \, 0\le j\le n-m$ and $h\in \CC[x_{n-m+1},
\dots, x_n]$  such that
\begin{eqnarray} \label{11} h_0\cdot g\,+\, \sum_{1\le j\le n-m}h_j\cdot u_{i_j}\,=\,h(x_{n-m+1},
\dots, x_n).
\end{eqnarray}
Moreover, we can rewrite the latter equality over the field $\CC(x_{n-m+1},
\dots, x_n)$ in the form
\begin{eqnarray} \label{12}
  \tilde{h}_0\cdot g\,+\, \sum_{1\le j\le n-m} \tilde{h}_j\cdot u_{i_j}\,=\, 1,
\end{eqnarray}
where $\tilde{h}_j=\frac{h_i}{h}\in \CC(x_{n-m+1},\ldots,x_n)[x_1,\ldots,x_{n-m}]$, with the common denominator in  $\CC[x_{n-m+1},\ldots,x_n]$ for  $0\le j\le n-m.$

We apply to  $(\ref{12})$ the Effective Nullstellensatz (see
e.g. \cite{Brownawell}, \cite{Heintz}, \cite{Kollar}, \cite{Jelonek}). This gives us  the bound
$d^{O(n)}$ on the degrees of
$\tilde{h_i}=h_i/h$ with respect to variables $x_1,\dots,x_{n-m}$  (for certain solutions) . To find $h_i/h$ one can solve
 the latter equality considering it as a linear system
over the field $\CC(x_{n-m+1},
\dots, x_n)$.

The algorithm can find $$\tilde{h_j}=\sum a_{I,j}x^I$$  with indeterminates $a_{I,j} \in  \CC(x_{n-m+1},\ldots,x_n)$, and monomials $x^I=x_1^{i_1}\cdot\ldots\cdot x_{n-m}^{i_{n-m}}$ with degrees  $i_1+\ldots+ i_{n-m}\leq d^{O(n)} $ substituting  $\tilde{h_j}$ in (\ref{12}) and  solving linear system over the field  $\CC(x_{n-m+1},\ldots,x_n)$. Clearing the common denominator  in (\ref{12}) gives
(\ref{11})  with  $$deg(h), deg(h_j) \leq d^{O(n^2)}, \quad 1\le j\le n-m $$

Now we have to estimate the maximal multiplicity  $\ord_x(g_{|X})$ for $g\in \CC[x_1,\ldots,x_n]$,  such that $deg(g)\leq d$ and  $x\in X$. We use (\ref{11}).
We get immediately by the above:

\begin{lemma}\label{boundM} The maximal multiplicity   $\ord_x(g_{|X})$ on $X$ for any  function $g\in \CC[x_1,\ldots,x_n]$,  such that $deg(g)\leq d$ and  $x\in X$, is bounded by the function   $M(d,n)= d^{O(n^2)}$ constructed as above:
$$ \ord_x(g_{|X})\leq M(d,n).$$
\end{lemma}

\begin{proof}
$\ord_x(g_{|X})\leq  \ord_x(h_0\cdot g)_{|X}=\ord_x h(x_{n-m+1},
\dots, x_n)_{|X}=\ord_x h(x_{n-m+1},\dots, x_n)\leq deg(h)\leq  M(d,n)= d^{O(n^2)}$.
\end{proof}

\subsubsection{\bf Derivations  on the subvariety $X\subset \CC^n$.}

In order to follow the construction of the algorithm we use the language of derivations $\Der_X$ on $X$.
Since our $X$ is embedded into $\CC^n$ it is natural to  represent all objects on $X$ as the restriction of the relevant objects on $\CC^n$ to $X\subset \CC^n$. Unfortunately  the sheaf of derivations on $\CC^n$ does not restrict well to $X$: $$\Der_{{\CC^n}|X}\neq  \Der_X.$$
Instead we consider 
$$\Der_{{\CC^n,X}}:=\{D \in \Der_{{\CC^n}} \mid D(\cI_X)\subset (\cI_X)\}.$$

\begin{lemma}
Let $u_1=0\ldots,u_{n-m}=0$ describe  $X\subset {\bf C^n}$. Then the ring
$\Der_{{\CC^n,X}}$ is generated  by
$$ \{u_i\cdot
d_{u_j}\}_{1\le i\le n-m,\, 1\le j\le n-m} \cup \{d_{u_j}\}_{n-m<j\le n}.$$
In particular the restriction $\Der_{{\CC^n,X}|X}=\Der_X$  is generated by $$ \{d_{u_j}\}_{n-m<j\le n}.$$
\end{lemma}

\begin{proof} This follows from the definition.
\end{proof}

Note that  since we have
$$(d_{u_j})_j=((\partial u_j/\partial x_i)_{i,j})^{-1}\cdot (d_{x_i})_i= \frac{1}{det((\partial u_j/\partial x_i)_{i,j})} adj((\partial u_j/\partial x_i)_{i,j})\cdot (d_{x_i})_i$$
we immediately get the following:

\begin{lemma} Let $\cU:=\{x\in \CC^n \mid det((\partial u_j/\partial x_i)_{i,j})\neq 0\}$.
The sheaf $\Der_{{\CC^n,X}}$ is generated over $U\subset \CC^n$ by
\begin{eqnarray} \label{*} \{u_i\cdot
d_{u_j}'\}_{1\le i\le n-m,\, 1\le j\le n-m} \cup \{d_{u_j}'\}_{n-m<j\le n},
\end{eqnarray}
 where
$$(d_{u_j}')_j= adj((\partial u_j/\partial x_i)_{i,j})\cdot (d_{x_i})_i .$$
\end{lemma}

The  derivations (\ref{*}) generate a subsheaf $\overline{\Der_{{\CC^n,X}}}$ of   $\Der_{{\CC^n,X}}$  over  $\CC^n$.  Both sheaves coincide over $U$.  Thus we will replace $\Der_{{\CC^n,X}}$ with
$\overline{\Der_{{\CC^n,X}}}$ for our computations over $U$.
\begin{lemma} \label{degree1} Let $\cI$ be any ideal on ${\bf C }^n$. Assume the maximal degree of some generating set of $\cI$ is  $\leq d_1$, and the maximal degree of $u_i$ is less than $d_2$. Then the maximal degree of generators of $\overline{\Der_{{\CC^n,X}}}(\cI)$ is  bounded by $ d_1+nd_2$.
\end{lemma}


\subsubsection{ Construction of  the coefficient homogenized companion ideal}

Recall that in step $2$ for the marked ideal $(\cI,\mu)$,  we find the maximal multiplicity
$\bar{\mu}\leq M(d,n)$ of the nonmonomial part $\cN(\cI)$, and construct the 
companion ideal $\cO(\cI)$, for which we immediately take homogenized coefficient ideal
$\cJ:=\cC(\cH(\cO(\cI)))$.
In our situation of the set $X\subset \CC^n$  defined by set of parameters $\{u_i\}_{1\le i\le n-m}$ on the open set $U\subset \CC^n$ we will use $\overline{\Der_{{\CC^n,X}}}$ instead of $\Der_X$ for the definition above.
Immediately from the definition, we get
a formula for  a bound $A(d,n,\mu)$ on the degrees of generators of the marked ideal
$\cC (\cH (\cI_q))$. Note first that we have the
 the following bounds on the multiplicities:

\begin{lemma}\label{multiplicity}
$\mu(\cN(\cI))=\bar{\mu}$, \quad   $\mu(\cO(\cI))\le \mu \cdot \bar\mu$ \quad  and $\mu(\cJ)\le (\mu \cdot \bar\mu)! \leq (\mu \cdot M(d,n))!\,.$
\end{lemma}

As a Corollary, we obtain:

\begin{lemma}\label{boundA} The maximal degree of generators of
$\cJ:=\cC(\cH(\cO(\cI)))$ is bounded by $$A(d,n,\mu):=(\mu\cdot \bar{\mu})!nd \leq (\mu \cdot  M(d,n))! nd \le (d^{O(n^2)})!\,.$$
\end{lemma}

\begin{proof} This follows from Lemma \ref{degree1}.
\end{proof}

\subsubsection{\bf Restriction to hypersurfaces of maximal contact, and to exceptional divisors}

In  step 1, we restrict $\cI$ to intersections of the exceptional divisors and maximal contact.

We need to estimate a bound $B(d,n,\mu)$ for the degree of the maximal contact $u\in \overline{\Der_{{\CC^n,X}}}^{\bar{\mu}-1}(\cN(\cI))$. The following is an immediate
consequence of Lemma \ref{degree1}:

 \begin{lemma}\label{boundB} The maximal degree of any maximal contact
$u\in \overline{\Der_{{\CC^n,X}}}^{\bar{\mu}-1}(\cN(\cI))$ is bounded by $$B(d,n,\mu):=\bar{\mu}nd
\leq  M(d,n)\cdot nd\le d^{O(n^2)}.$$

\end{lemma}

\subsubsection{A bound for the number of generators  of $J$}

First, we state the basic properties concerning the number of generators  of the ideal in 
the following lemma:

\begin{lemma}
\begin{enumerate} \item The number of generators of $\overline{\Der_{\CC^n,X}}(\cI)$ is given
by $(n+1)l(\cI)$.
\item The number of generators of $\overline{\Der_{\CC^n,X}}^{i}(\cI)$ is given
$(n+1)^il(\cI)$.
\item The number of generators of $\cI^i$ is bounded  by $l(\cI)^i$.
\item The number of generators of $\cI'=\cO(\cI)$ can be bounded  by $L_{\cO}(l(\cI),\mu):=l(\cI)^\mu+1$.
\item The number of generators of $\cI'=\cH(\cI,\mu)$ can be bounded  by $L_{\cH}(l(\cI),\mu):=\mu (n+1)^{\mu^2} l(\cI)^\mu
$.
\item The number of generators of $\cI'=\cC(\cI,\mu)$ can be bounded  by $L_{\cC}(l(\cI),\mu)):=\mu (n+1)^{\mu !} l(\cI)^{\mu !}
$.
\end{enumerate}
\end{lemma}

\begin{proof} Immediate from the definitions.
\end{proof}

 By the preceding lemma, we get
 \begin{align*}
 l(\cH(\cO(\cI))) &\leq  L_{\cH}(L_{\cO}(l(\cI),\mu),\mu\cdot M(d,n)\,,\\
 l(\cJ)=l(\cC(\cH(\cO(\cI)))) &\leq  L_\cC(L_{\cH}(L_{\cO}(l(\cI),\mu),\mu\cdot M(d,n)),\mu\cdot M(d,n))\,.
 \end{align*}
Thus we get:

\begin{corollary}\label{number} $l(\cJ)\leq F(d,n,\mu)$, where $$F(d,n,\mu):=L_\cC(L_{\cH}(L_{\cO}(l(\cI),\mu),\mu\cdot M(d,n))),\mu\cdot M(d,n)).$$
\end{corollary}

\begin{remark}
The algorithm of \cite{B-M5} does not involve the homogenization step and therefore
gives better estimates for the elementary functions introduced. In particular:
\begin{enumerate}
\item The degree of  generators of $\cJ$ is bounded by $B(d,n,\mu)$ (which improves the bound $A(d,n,\mu)$, cf. Lemma \ref{boundA}).
\item The number of  generators
$l(\cJ)$ can be bounded by $L_\cC (L_{\cO}(l(\cI),\mu),\mu\cdot M(d,n))$
(which improves the bound $F(d,n,\mu)$; cf. Corollary \ref{number}),
\end{enumerate}
However, the above improvements do not affect the overall Grzegorczyk complexity class 
$\cE^{m+3}$. (See Theorem \ref{main3}.)
\end{remark}

Summarizing:

\begin{lemma} \label{2a1} The effect  of passing from $\cI$ to $\cJ=\cC(\cH(\cO(\cI)))$ as in Step 2a/Step1 can be described by the function $$\Delta_{2a}(r,m,d,n,l,q,\mu):=(r,m,A(d,n,\mu),n,F(d,n,\mu,l),q, (\mu\cdot M(d,n))!).$$
\end{lemma}

\subsubsection{A bound for the number of maximal contacts and the relevant neighborhoods}
We will construct maximal contacts along with  the open neighborhoods for which they are defined.
Each maximal contact $u\in \overline{\Der_{\CC^n,X}}^{\bar{\mu}-1}(\cN(\cI))$ that
we consider is of the form $u=D^{\bar{\mu}-1}(g_i)$, where
$D^{\bar{\mu}-1}=D_1^{a_1}\ldots D_n^{a_n}$ is a certain composition of ${\bar{\mu}-1} $   differential operators (\ref{*})
(i.e., the $D_i$ are of the form in  (\ref{*}), and $a_1+\ldots+a_n=\bar{\mu}-1$).
Consider all differential operators $\{D_r^{\bar{\mu}}\}_{r\in R}$,  which are certain compositions of 
${\bar{\mu}}$   differential operators    (\ref{*}), and take all the corresponding functions $f_{r,i}:=D_r^{\bar{\mu}}(g_i)$. On the open set $U_{r,i}=U\setminus V(f_{r,i})$, consider the maximal contact $u_{r,i}=D^{\bar{\mu}-1}(g_i)$, where
 $D^{\bar{\mu}-1}=D_1^{a_1}\ldots D_n^{a_n}$
  is a certain composition   of ${\bar{\mu}-1}
$   differential operators (\ref{*}) obtained from $D_r^{\bar{\mu}}=D_1^{b_1}\ldots D_n^{b_n}$ by replacing one of the positive $b_i$ with $b_i-1$ (i.e $a_i:=b_i-1$ for some $b_i>0$ and $a_j=b_j$ for $j\neq i$.)

\begin{lemma} The number of the maximal contacts $u_{r,i}\in \overline{\Der_{{\CC^n,X}}}^{\bar{\mu}-1}(\cN(\cI))$ and at the same time the number of neighborhoods $U_{i,r}\subset U$ can be bounded
by  $$ C(d,n,\mu)\cdot l(\cI),$$
where $l(\cI)$ is the number of generators of $\cI$, and
$$
  C(d,n,\mu):=  {{M(d,n)+n}\choose{n}}\,.
$$
\end{lemma}

\begin{proof}  The number of the maximal contacts is bounded by ${ { {\bar{\mu}}+n} \choose{n}}\cdot l(\cI) \, \leq \, C(d,n,\mu)\cdot l(\cI), where
  C(d,n,\mu):=  {{M(d,n)+n}\choose{n}}\,.$
\end{proof}


Summarizing:

\begin{lemma} \label{1ab}
The effect  of passing from $\cI$ to $\cJ=\cC(\cH(\cO(\cI)))$ and then to $\cJ_{|H_s}$ in Step 1a or to  $\cJ_{|V(u)}$ as in the Step 1b can be described by the function $$\Delta_{1}(r,m,d,n,l,q,\mu):=
(r,m-1,A(d,n,\mu),n,F(d,n,\mu,l),q\cdot l \cdot C(d,n,\mu), (\mu\cdot M(d,n))!)\,.$$
\end{lemma}

Note that the restriction to the maximal contact does not affect the degree since the function $B(d,n,\mu)$ measuring the degree of maximal contact is smaller than $A(d,n,\mu)$.

\begin{remark} The particular form of the bounds obtained
does not strongly influence Theorem ~\ref{main3}; we need only that the functions
 belong to the class $\cE^3$. (See the beginning of the next section.)
\end{remark}

\section{Complexity bound of the resolution algorithm in terms of Grzegorczyk's classes}

\subsection{Language of Grzegorczyk's classes}
The complexity estimate of the desingularization algorithm which we provide in this section is given in terms of 
Grzegorczyk's classes $\cE^l,\, l\ge 0$ of
primitive-recursive functions \cite{Gr}, \cite{Wagner}.
To make the paper self-contained, we provide a definition of $\cE^l$ by induction
on $l$
(informally speaking, $\cE^l$ consists of integer functions $\ZZ^s\to \ZZ^t$ whose
construction requires $l$ nested primitive recursions).

For the base definition, $\cE^0$ contains constant functions $x_k\mapsto c$,
functions $x_k\mapsto x_k+c$ and
projections $(x_1,\dots,x_n)\mapsto x_k$ for any variables $x_1,\dots,x_n$.

The class $\cE^1$ contains linear functions $x_k\mapsto c\cdot x_k$ and $(x_{k_1},
x_{k_2})\mapsto x_{k_1}+x_{k_2}$.

The class $\cE^2$ contains all polynomials with integer coefficients

Let $l\ge 2$. For the inductive step of the definition, assume that functions $G(x_1,\dots,x_n),
H(x_1,\dots,x_n,y,z)\in \cE^l$. Then the function $F(x_1,\dots,x_n,y)$ defined by the
primitive recursion,
  \begin{eqnarray}\label{grz1}
F(x_1,\dots,x_n,0)=G(x_1,\dots,x_n), \end{eqnarray}
 \begin{eqnarray}\label{grz2}
F(x_1,\dots,x_n,y+1)=H(x_1,\dots,x_n,y,F(x_1,\dots,x_n,y)), \end{eqnarray}
belongs to $\cE^{l+1}$.

To complete the definition of $\cE^l,\, l\ge 0$,
take the closure with respect to composition and the following {\it limited primitive recursion}:
\begin{itemize}
\item[] Let $G(x_1,\dots,x_n), H(x_1,\dots,x_n,y,z), Q(x_1,\dots,x_n,y) \in \cE^l$. Then the function $F(x_1,\dots,x_n,y)$ defined by (\ref{grz1}),(\ref{grz2}) also belongs to  $ \cE^l$, {\it provided that}
$F(x_1,\dots,x_n,y)\leq Q(x_1,\dots,x_n,y)$.
\end{itemize}
Clearly, $\cE^{l+1}\supset \cE^l$.

Observe that 
$\cE^3$
contains all  towers of exponential functions and
$\cE^4$
contains all tetration  functions \cite{Gr}, \cite{Wagner}.

The
union $\cup_{l<\infty}
\cE^l$ coincides with the set of all primitive-recursive functions.

\subsection{Resolution algorithm as a graph}
 It is instructive to
represent the resolution algorithm in the form of a tree $T$ as in the following
Figure.

\begin{figure}[htbp]
\centerline{\epsfig{file=tree2.eps,width=.7\textwidth}}
\caption{\label{tree2}}
      \end{figure}

Each node $a$ of $T$ corresponds to an intermediate object $T_a=(\{X_{\alpha,\beta},\cI_{\alpha,\beta},E_{\alpha,\beta},U_{\alpha,\beta}\}_{\alpha,\beta},
n,\mu)$. Each node $a$ is labeled either by 1 or 2 depending on whether it corresponds to
step 1 or 2 in the description of the algorithm (see the previous sections). An edge
from a node labeled by 1 leads to its child node labeled by 2 and the edge is labeled in its turn
either by 2a or by 2b depending on the step to which it corresponds. Similarly, an
edge from a node labeled by 2 leads to its child node labeled by 1 and is labeled in
turn either by 1a or by 1b. In the Figure a child node is always located to the right
from a node.

The algorithm yields $T$ by recursion starting with its root. Assume that $a$ and $T_a$
are already constructed. The next task of the algorithm is to resolve the object $T_a$. To
this end
the algorithm first constructs the child nodes $a_1,\dots,a_{\overline t}$ of $a$
according to the algorithm. The order of producing $a_1,\dots,a_{\overline t}$ goes
from up to down in the Figure. The algorithm resolves the objects
$T_{a_1},\dots,T_{a_{t-1}}$ by recursion on $0\le t\le \overline t$ and in the process
modifies $T_a:=T_a(0)$, obtaining the current object $T_a(t-1)$. Then the algorithm yields
$a_t,\, T_{a_t}$, resolves $T_{a_t}$ and
collects all the blow ups produced while resolving $T_{a_t}$ and applies them (with the
same centers) to
the current object $T_a(t-1)$; the resulting object we denote by $T_a(t)$. This
allows the algorithm to yield a child node $a_{t+1}$ and $T_{a_{t+1}}$ following the
description from the previous sections.

For the leaves of $T$ there are two possibilities: either a leaf is labeled by 2 or a certain
node $a$
labeled by 2 could have a single edge (the lowest in the Figure among the edges originating
at $a$) labeled by 2b which leads to a child node of $a$ being a leaf corresponding to the
monomial case (labeled by $M$). Note also
that if $a$ is labeled by 1 then the top few edges originating at $a$ are labeled by 1a and the
remaining bottom ones are labeled by 1b (in the order from up to down in the Figure).

Observe that the dimension of the varieties $X_{\alpha,\beta}$ corresponding to $a$ drops while passing
to any of its child nodes when $a$ is labeled by 1, and the dimension does not increase
when $a$ is labeled by 2. Therefore, the depth of $T$ does not exceed $2m$.

\subsection{Main recursive functions}
Now we proceed to the bounds of some recursive functions related to the ingredients of $\cT$. Set  $$\gamma:=(r,m,d,n,l,q,\mu) \in {\bf Z}^7_{\geq 0}\,. $$
Let $\cT_{*}$ be the canonical resolution of $\cT$.
For simpilicity of notation, we introduce the following function defined on ${\bf Z}^7_{\geq 0}$:
$$\Gamma^{(m)}(\gamma):=(r+{R^{(m)}}(\gamma),m,{D^{(m)}}(\gamma), {N^{(m)}}({\gamma}), {L^{(m)}}(\gamma), {Q^{(m)}}(\gamma),\mu)\,,$$
where the subscript $r$ can be interpreted as the subscript in the resolution of $\cT$, and
\begin{enumerate}
\item ${R^{(m)}}(\gamma)$ is the number of  blow-ups needed to resolve the initial marked ideal with data bounded by $(m,d,n,l,q,\mu)$;
\item  ${D^{(m)}}(\gamma)$ is a function bounding the maximum of the degrees of
all the polynomials which represent $\cT_{*}$ and all objects constructed along the way (in particular, the centers);
\item ${N^{(m)}}(\gamma)$ is a bound for the dimensions of the ambient affine spaces  constructed along the way;
\item ${L^{(m)}}(\gamma)$ is a bound for the number of polynomials appearing in the
description of a single  neighborhood $U_{\alpha,\beta}$ on resolving
$T$;
\item ${Q^{(m)}}(\gamma)$ is the number of neighborhoods in all the auxiliary objects (in particular, the centers)
appearing on resolving
$T$.
\end{enumerate}

\begin{remark} The functions ${R^{(m)}}(\gamma)$, ${D^{(m)}}(\gamma)$, ${N^{(m)}}(\gamma)$ do not depend on $l,q$.
\end{remark}

\subsubsection{Algorithm revisited}

Let  $(\cI,\mu)$ be  a marked ideal on an $m$-dimensional smooth variety $X$. Consider the corresponding object $$\cT^{(m)}=(\{X_{\alpha,\beta},\cI_{\alpha,\beta},E_{\alpha,\beta},U_{\alpha,\beta},\CC^n_\alpha \mid \alpha \in A, \beta \in B\},\mu)$$ with the initial data $$\gamma:=(0,m,d,n,l,q,\mu).$$
Our next goal is two-fold. We will give recursive formulas for $\Gamma^{(m)},{R^{(m)}},
{D^{(m)}},{N^{(m)}},{L^{(m)}},{Q^{(m)}}$ and   prove by induction on $m$ these functions  belong to
Grzegorczyk's class $\cE^{m+3}$. In the base of the induction (i.e., for $m=0$), the functions
$${R^{(0)}}=1,\,{D^{(0)}}=O(dn),\,{N^{(0)}}\le 2n,\,
{L^{(0)}}\le l\cdot (dn)^{O(n)},\,{Q^{(0)}}\le nq$$ belong to the class $\cE^3$,
as does $\Gamma^{(0)}$.

Now we proceed to the inductive step.  Assume that $\Gamma^{(m-1)},{R^{(m-1)}},
{D^{(m-1)}},{N^{(m-1)}},{L^{(m-1)}},{Q^{(m-1)}}$  belong to
Grzegorczyk's class $\cE^{m+2}$, where $m\geq 1$.

If $\cI=0$, then the resolution is done by the single blow-up at the center $C=X$ and the object $\cT^{(m)}$ is transformed into an object $\cT^{(m)}_1$ with $X_1=\emptyset $ and data bounded by $Bl(\gamma)\in \cE^3$ (see Lemma \ref{boundG}).

If   $\cI\neq 0$, then the resolution algorithm can be represented by the following scheme.
\medskip

{\bf Step 2}.  Resolve $(\cI,\mu)$ on the $m$-dimensional smooth variety $X$. Consider the corresponding object $\cT^{(m)}$ with initial data $\gamma:=(0,m,d,n,l,q,\mu).$
Let $\bar{\mu}$ denote the maximal order of $\cN(\cI)$ on $X$. We have the following estimate for $\bar{\mu}$:  $$\bar{\mu} \leq M(d,n) \in \cE^3$$
(cf Lemma \ref{boundM}).
\medskip

{\bf Step 2a.} In this Step we are going to decrease
 the maximal order of the nonmonomial part $\cN(\cI)$ by resolving the companion
ideal $O(\cI,\mu)$. In fact we perform an additional modification of $\cO(\cI)$ and construct the ideal $\cJ:=\cC(\cH(\cO(\cI))).$
This corresponds to  the new object
$\cT_{1}^{(m)}$ with the initial data $$\gamma^{(2a)}:=\Delta_{2a}(\gamma) \in \cE^3,$$
(see Lemma \ref{2a1}).

The object $\cT_{1}^{(m)}$ will be then resolved  and its resolution will cause the maximal order to decrease. The resolution $\cT_{1}^{(m)}$ is done by performing Step 1.
\medskip

{\bf Step 1}. In this step we resolve $\cJ$, i.e., $\cT_{1}^{(m)}$.
The Step splits into two Steps (1a) and (1b).
\medskip

{\bf Step 1a}.  Move apart all unions of the intersections $H^s_\alpha\subset \CC^n_\alpha$
of $s$ divisors in $E$, where $s$ is the maximal number of divisors in $E$ through
points in $\supp\,\cJ = \supp(\cJ,\mu(\cJ))$.
For any $\alpha$,  resolve all ${\cJ}_{|H^s_\alpha}$.  We construct a new object  $$\cT_2^{(s)}:=(\{H^s_{\alpha,\beta},\cJ_{\alpha,\beta},E_{\alpha,\beta},U_{\alpha,\beta},\CC^n_\alpha \mid \alpha \in A, \beta \in B\},\mu(\cJ)),$$
with initial data bounded by $$\gamma^{(1a)}:=\Delta_1(\gamma) \in \cE^3 $$
with $s\leq m-1$ (see Lemma \ref{1ab}).

By the inductive assumption,  the resolution  of $\cT_2^{(s)}$, i.e., the sequence  
$\cT_{2*}^{(s)}$ of the induced intermediate objects determined by the blow-ups,  requires at most ${R^{(m-1)}}(\gamma^{(1)})$ blow ups. The maximal degree of the polynomials of the centers and the objects $\cT_{2*}^{(s)}$ describing the resolution does not exceed ${D^{(m-1)}}(\gamma^{(1)})$. The dimension $n$
of the objects  does not exceed ${N^{(m-1)}}(\gamma^{(1)})$.

Note that the resolution of $\cT^{(s)}_{2*}$  determines a multiple blow-up  $\cT^{(m)}_{1*}$ of  $\cT^{(m)}_{1}$ consisting of ${R^{(m-1)}}(\gamma^{(1)})$ blow-ups. We have a direct correspondence between objects $\cT^{(m)}_{1*}$, and $\cT^{(s)}_{2*}$. The bound $$\Gamma^{(m-1)}(\Delta_1(\gamma)) \in \cE^{m+2},$$
for the data for the resolution $\cT^{(s)}_{2*}$,
given by the induction, remains valid  for the data for $\cT^{(m)}_{1*}$ as we use the same centers, the same ambient affine spaces, etc., for these multiple blow-ups. Only the strict transforms of the current $X$ are different, and this does not affect the bounds for the data. We have additional  equations to describe the current $X$ in $\cT^{(s)}_{2*}$, as compared to those in $\cT^{(m)}_{1*}$.

{Step 1a} is performed at most $s\leq m$ times. Introduce the auxillary unknown $t=0,1,\ldots,m,$  and the function $\Gamma^{(m)}_{1a}(\gamma,t)$ which measures the possible effect after performing Step 1a $t$ times:
\begin{align*}\Gamma^{(m)}_{1a}(\gamma,0)&:=\Delta_{1}(\gamma) \in \cE^3\,, \\
\Gamma^{(m)}_{1a}(\gamma,t+1)&:=\Gamma^{(m-1)}(\Gamma^{(m)}_{1a}(\gamma,t)).
\end{align*}

Since the Step 1a is performed at most $m$ times, its final effect after completing Step 1a and passing to Step 1b is measured by the function
$${\Gamma^{(m)}_{1b}}(\gamma):=\Gamma^{(m)}_{1a}(\gamma,m)\,.$$
Note that for any fixed value of $t=t_0$ (in particular, for $t=m$) the functions
$\Gamma^{(m)}_{1a}(\gamma,t_0),$ belong to the class $\cE^{m+2}$ due to the
inductive hypothesis and Lemma~\ref{boundG}, Corollary~\ref{number}. Therefore,
${\Gamma^{(m)}_{1b}}$ belongs to the class $\cE^{m+2}$.
(We use here the property that Grzegorczyk classes are closed under composition.)
After performing Step 1a, we
have moved apart all strict transforms of $E$ and $\supp\,\cJ = \supp(\cJ,\mu(\cJ))$.
\medskip

{\bf Step 1b}. If  the strict transforms of $E$ do not intersect  $\supp\,\cJ$, we resolve 
$\cJ$, i.e., the object  $\cT^{(m)}_{1}$. This is achieved
by resolving
$\cJ_{|V(u)}$ (by induction), where $V(u)$ is a hypersurface of
maximal contact. After completing Step 1a, the bound $\gamma^{(1a)}$ is transformed to $$
{\gamma^{(1b)}}:={\Gamma^{(m)}_{1b}}(\gamma)=({r^{(1b)}},m,{d^{({1b})}}, {n^{(1b)}},{l^{(1b)}}, {q^{(1b)}}, (\mu\cdot(M(d,n)))!)$$
(cf. Lemma~\ref{multiplicity}).

Passing from $\cJ$ to  ${\cJ}_{|V(u)}$, we adjoin the equations of maximal contact as well as create new neighborhoods. This operation has been reflected in the construction of $\Delta_1(\gamma)$.
By the construction of ${\Gamma^{(m)}_{1b}}(\gamma)$ and $\Delta_1(\gamma)$, the degree of the maximal contact does not exceed ${d^{({1b})}}$, while the number of neighborhoods does not exceed ${q^{(1b)}}$. In other words ${\Gamma^{(m)}_{1b}}(\gamma)$ bounds the initial data for ${\cJ}_{|V(u)}$.

The resolution process for ${\cJ}_{|V(u)}$ leads eventually to resolution of the object 
$\cT_1^{(m)}$ corresponding to ${\cJ}$, with the data bounded by
$${\Gamma^{(m)}_{1}}(\gamma):=\Gamma^{(m-1)}({\Gamma^{(m)}_{1b}}(\gamma))= \Gamma^{(m)}_{1a}(\gamma,m+1)= ({r^{(1)}},m,{d^{({1})}}, {n^{(1)}},{l^{(1)}}, {q^{(1)}},(\mu\cdot(M(d,n)))!),$$
for the relevant ${r^{(1)}},{d^{({1})}}, {n^{(1)}},{l^{(1)}}, {q^{(1)}}$.
Hence the function ${\Gamma^{(m)}_{1}}$ belongs to class $\cE^{m+2}$ by the
inductive hypothesis and Lemma~\ref{boundG}, Corollary~\ref{number}.
This completes Step 1.

The object $\cT^{(m)}$ corresponding to $\cI$ with initial data $\gamma$ is transformed to the new object with the data bounded by
$${\Gamma^{(m)}_{2a}}(\gamma)
:=({r^{(1)}},m,{d^{({1})}}, {n^{(1)}},{l^{(1)}}, {q^{(1)}},\bar{\mu})$$
with smaller  $\bar{\mu}$ ---  the maximal order of the new
$\cN(\cI)$. (Note that $\bar{\mu}< M(d,n)$).

This completes Step 2a.
This Step 2a is then repeated at most $M(d,n)$ times until the maximal order drops to zero 
when we arrived at the monomial case.
The final effect of Step 2a is measured then by the recursive function
\begin{align*}
\Gamma^{(m)}_{2a}(\gamma,0)&=\gamma\,,\\
\Gamma^{(m)}_{2a}(\gamma,t+1)&={\Gamma^{(m)}_{2a}}(\Gamma^{(m)}_{2a}(\gamma,t)).
\end{align*}
Therefore, the function $\Gamma^{(m)}_{2a}$ belongs to class $\cE^{m+3}$, by the definition
of Grzegorczyk classes (see (\ref{grz1}),(\ref{grz2})),  and by Lemmas~\ref{boundG},~\ref{boundM}.

Putting $t=M(d,n)$ gives the final effect after completing all necessary Steps 2a and subsequently  passing to
Step 2b:
$${\Gamma^{(m)}_{2b}}(\gamma):=\Gamma^{(m)}_{2a}(\gamma,M(d,n));$$
thereby, the function ${\Gamma^{(m)}_{2b}}$ belongs to class $\cE^{m+3}$ as well.

The procedure eventually reduces $(\cI,\mu)$ to the monomial marked
ideal $\cI=\cM(\cI)$.
\medskip

{\bf Step 2b}. Resolve the monomial marked ideal $\cI=\cM(\cI)$.
The marked ideal corresponds to the object  $\cT^{(m)}$ with data $$({r^{(2b)}},m,{d^{({2b})}}, {n^{(2b)}},{l^{(2b)}}, {q^{(2b)}},\mu)=:{\Gamma^{(m)}_{2b}}(\gamma).$$
The resolution of $\cI=(x^\alpha)$ consists of blow-ups each of which decreases the multiplicity $|x^\alpha| \leq d^{({2b})}$. The resolution of $\cI$ requires at most $d^{({2b})}$ blow-ups. Thus the final solution data can be bounded by the function
$$\Gamma^{(m)}(\gamma):=\overline{Bl}( {\Gamma^{(m)}_{2b}}(\gamma),{d^{(2b)}}) \in \cE^{m+3}\,.$$

We summarize the bounds achieved in the following Corollary (recall that the notation can be
found in subsection \ref{4.1.1}).

\begin{corollary}\label{total}
When resolving a marked ideal $(X,\cI,E,\mu)$ on $X\subset \CC^n$ by the Hironaka algorithm,
the degree $d$, the number $l$ of the polynomials occurring, the embedding
dimension $n$, the
number $r$ of the blow-ups and the number $q$ of the affine
neighborhoods satisfy, for fixed $m=\dim\,X$,  the recursive equalities above, and are majorized by a function
$$(r,m,d,n,l,q,\mu):= {\Gamma^{(m)}}(0,m,d_0,n_0,l_0,q_0,\mu)\in \cE^{m+3},$$
for the initial values $d=d_0$,  $n=n_0$,  $l=l_0$,  $q=q_0$.
\end{corollary}

\subsection{Complexity of the algorithm}

The principal complexity result of the paper is the following assertion.

\begin{theorem}\label{main2}
When resolving a marked ideal $(X,\cI,E,\mu)$ on $X\subset \CC^n$ by the Hironaka algorithm, its complexity can be bounded, for fixed $m=\dim\,X$,  by $$b^{O(1)}\cdot \cF^{(m)}(d_0,n_0,l_0,q_0, \mu),$$ for
a certain function $\cF^{(m)}(d_0,n_0,l_0,q_0,\mu)\in \cE^{m+3}$.
\end{theorem}

{\bf Proof.} Indeed, each step of the algorithm consists of solving a certain subroutine (basically,
solving a linear system) over the coefficients of the current polynomials. Therefore,
Corollary~\ref{total} provides a bound on the number of arithmetic operations with the
coefficients of the current polynomials (providing a function from class $\cE^{m+3}$). On the
other hand, all the coefficients of the polynomials for the next step are obtained as
results of these arithmetic operations, so the bit sizes of the coefficients and the complexity   grow by at most
$b\cdot \cF^{(m)}$, for a function $\cF^{(m)}={\mathcal G}(\Gamma^{(m)}(0,m,d_0,n_0,l_0,q_0,\mu))\in \cE^{m+3}$ with a suitable function $\mathcal G \in \cE^3$.  
The cost of a single
arithmetic operation is obviously polynomial.

As a corollary we obtain the following theorem.

\begin{theorem} \label{main3}
When either
\begin{enumerate}
\item resolving singularities
of $X\subset \CC^{n_0}$, 
\item[{\it or} (2)] principalizing an ideal sheaf $\cI$ on  a nonsingular $X\subset \CC^{n_0}$,
\end{enumerate}
by the Hironaka algorithm, the degree $d$, the number $l$ of the polynomials occurring, the embedding
dimension $n$, the
number $r$ of the blow-ups, and the number $q$ of the affine
neighborhoods satisfy, for fixed $m=\dim\,X$,  the recursive equalities above and are 
majorized by a function
$$(r,m,d,n,l,q,1):= {\Gamma^{(m)}}(0,m,d_0,n_0,l_0,q_0,1)\in \cE^{m+3},$$\noindent for the initial values $d=d_0$,  $n=n_0$,  $l=l_0$,  $q=q_0$.

The complexity of the algorithm is bounded by $$b^{O(1)}\cdot \cF^{(m)}(d_0,n_0,l_0,q_0,\mu),$$ for a certain function $\cF^{(m)}(d_0,n_0,l_0,q_0,\mu)\in \cE^{m+3}$.
\end{theorem}

\begin{remark}
In the proof above, we gave a more explicit form of $\cF^{(m)}$ providing additional information
on its dependance on $r,d,n,l,q,\mu$. But the main consequence of the Theorem is that $m=
\dim\,X$ provides the most significant contribution to the complexity bound.
\end{remark}
\bigskip

\section{Appendix. Applications to positive characteristic.}

Define $D(m,d_0,n_0,l_0,q_0,\mu)$ and $N(m,d_0,n_0,l_0,q_0,\mu)$ to be the 
$d-$ and $n-$coordinates of  \\ ${\Gamma^{(m)}}(0,m,d_0,n_0,l_0,q_0,\mu)$
(in future considerations we drop subscripts in the above presentations).

Observe that  we can count the maximal order of all ideals occuring in single neighborhoods representing our data; that is, the marked ideal $(X,\cI,E,\mu)$ embedded in the affine space $\bA^n=\CC^n$ .

It follows from Lemma \ref{boundM} that this  order does not exceed
$$M(\bar{D},\bar{N})=M(m,d,n,l,\mu),$$ where $\bar{D}(m,d,n,l,\mu):={D}(n,d,n,l,1,1,1,\mu)$ and $\bar{N}(m,d,n,l,\mu):=N(n,d,n,l,1,1,1,\mu)$ are  functions in $(m,d,n,l,\mu)$. 
By the construction all the functions above are in $\cE^{n+3}$.

This gives us some rough control of the algorithm in characteristic $p$.

We can define $M(X,\cI,E,\mu, \bA^n):=M(m,d,n,l,\mu)$, where $(X,\cI,E,\mu)$ is a marked ideal
described  in some open subset $U$  of $\bA^n$ by $l$ polynomials of degree less than $n$.

Set  $M(X,\cI,E,\mu,x)$ to be the minimum of $M(X',\cI',E',\mu, \bA^{n'}),$ over all the affine neighborhoods $X'\subset X$ of $x$, embedded into possible affine spaces $\bA^{n'}$.
Note that $X'$ is locally closed and can always be shrunk so the data depend 
only on the stalk of the marked  ideal $\cI$ at $x \in X$. 
Similarly $\bar{M}(X,\cI,E,\mu,x)$ is the minimum of all the ${M}(X',\cI',E',\mu, x')$, for which there exists 
$(X'',\cI'',E'',\mu, x'')$ with \'etale morphisms $(X'',\cI'',E'',x'')\to (X,\cI,E, x)$ and $(X'',\cI'',E'', x'')\to (X',\cI',E',x')$.

Note that the function $\bar{M}(X,\cI,E,\mu,x)$ is  defined in any 
characteristic.
Moreover, we have:
\begin{lemma}
 The function $\bar{M}(X,\cI,E,\mu,x)$ has the following properties:
\begin{enumerate}
\item If $X\subset \bA^n$ is a locally closed subset of $\bA^n$ described in some open subset $U$  of $\bA^n$ by $l$ polynomials of degree less than $n$, and $x\in U$, then  $\bar{M}(X,\cI,x)\leq M(d,n,l)$.

\item If $(X',\cI',E',x')\to (X,\cI,E,x)$ is \'etale then $\bar{M}(X,\cI,E,\mu,x)=\bar{M}(X',\cI',E',\mu,x')$.
\item  The function $$\bar{M}(X,\cI,E,\mu):=\max_{x\in X}\{\bar{M}(X,\cI,E,\mu,x)\}$$
is well defined.

\item  Consider the canonical   Hironaka resolution of $(X, \cI, E,\mu)$. 
The multiplicities of the marked ideals occuring do not exceed $\bar{M}(X,\cI,E,\mu)$.
\end{enumerate}
\end{lemma}

\begin{proof} Follows from the definition. We use here also the canonicity of Hironaka resolution and the fact  that  \'etale morphisms preserve multiplicities.
\end{proof}

This implies the following.

\begin{theorem} There exists a canonical Hironaka resolution of marked ideals 
for which
$$\bar{M}(X,\cI,E,\mu)<p.$$
\end{theorem}

\begin{proof}
The algorithm is the same as in characteristic zero. We just need minor modifications of the basic results.
 We replace the hypothesis of characteristic zero with the one that the multiplicity $\mu$ in the relevant marked ideals is less than characteristic $p$ in Lemmas  \ref{le: Vi1}, \ref{Vi2}, \ref{Vi3}, \ref{le: Gi}, \ref{le: homo0},  \ref{le: homo}, \ref{le: S}, with proofs unchanged. 

Theorem \ref{main} is replaced with the following.

\begin{theorem}\label{main4} Assume the characteristic of base field is $p$. 
For any marked ideal $(X,{\cI},E,\mu)$ such that $\bar{M}(X,\cI,E,\mu)<p$,
there is an associated  resolution $(X_i)_{0\leq i\leq m_X}$, 
called \underline{canonical},
 satisfying the following conditions:
\begin{enumerate}
\item For any  surjective \'etale morphism $\phi: X '\to X$, the induced sequence
 $(X'_i)=\phi^*(X_i)$ is  the canonical resolution of $(X',{\cI}',E',\mu):=\phi^*(X,{\cI},E,\mu)$.
\item For any  \'etale morphism $\phi: M'\to M$, the induced sequence
 $(X'_i)=\phi^*(X_{i})$ is an extension   of the canonical resolution of  $(X',{\cI}',E',\mu):=\phi^*(X,{\cI},E,\mu)$.
\end{enumerate}
\end{theorem}
 \begin{proof} The proof is the same as before. We apply the same algorithm as 
in characteristic zero. The multiplicities of the marked ideals occuring
do not exceed $\bar{M}(X,\cI,E,\mu)<p$.
Indeed, it suffices to  verify this for a marked ideal in a neighborhood of 
some point $x\in X$.
We can find a marked ideal $(X',\cI',E',\mu)$ which is  \'etale equivalent to
 $(X,\cI,E,\mu)$ in a neighborhood of $x \in X$, and such that  
${M}(X',\cI,E',\mu)={M}(X',\cI,E',\mu,x')=\bar{M}(X,\cI,E,\mu,x)$. The 
algorithm for $(X',\cI,E',\mu,x')$ is identical to that in characteristic zero 
since it will not create marked ideals of the multiplicities greater than $p$. 
The algorithm for $(X,\cI,E,\mu,x)$ by the modified basic results above will 
create marked ideals with the same multiplicities.
\end{proof}
\end{proof}

\begin{remark} Note that the resolution is canonical and it is $G$-equivariant 
with respect to the action of any group $G$. Thus existence of a resolution 
over an algebraically closed field implies a resolution over any perfect field 
$K$. The reasoning is the same as in characteristic zero.
Consider the action of the Galois group $G:=Gal(\overline{K}/K)$. By the above,
all the centers are $G$-stable so they are defined over $K=\overline{K}^G$. 
(See, for instance, \cite{Wlod} for details.)
\end{remark}

\begin{corollary} Assume the base field is perfect of characteristic $p$.
There exists a function $M(d,n,l):=M(n,d,n,l,1) \in  \cE^{n+3}$   
(independent of characteristic) such that
for all $Y\subset \bA^n$ described  by $l$ polynomials of degree less than 
$n$, for which $M(d,n,l)<p$, there is a canonical (embedded) resolution of singularities.
\end{corollary}
\begin{remark} Note that in the above formulation the ambient nonsingular variety is $X=\bA^n$. That is why we put here $m=n=\dim(X)$.
\end{remark}
\begin{corollary} There is a canonical principlization  of ideals $\cI$ in 
$\bA^n$ described  by $l$ polynomials of degree less than $n$, for which
$M(d,n,l)<p$. 
\end{corollary}


\begin{thebibliography}{CC}


  \bibitem{A}  S. S. Abhyankar. Desingularization of plane curves. In Algebraic Geometry, Arcata 1981, Proc. Symp. Pure Appl. Math. 40. Amer. Math. Soc., 1983.

 \bibitem{A2}   S. S. Abhyankar. Good points of a hypersurface. Adv. in Math., 68:87-256, 1988.
 \bibitem{Abramovich-de-Jong} D. Abramovich and A. J. de Jong,
{\em  Smoothness, semistability, and toroidal geometry}, J. Alg. Geom. 6,
1997, p. 789-801.
\bibitem{Abramovich-Wang} D. Abramovich and J. Wang,
{\em  Equivariant resolution of singularities in characteristic 0},
Math. Res. Letters 4, 427-433 (1997).

\bibitem{AHV}    J. M. Aroca, H. Hironaka, and J. L. Vicente.{\em Theory of maximal contact. }Memo Math. del Inst. Jorge Juan, 29, 1975.

\bibitem{AHV2}    J. M. Aroca, H. Hironaka, and J. L. Vicente. {\em Desingularization theorems.} Memo Math. del Inst. Jorge Juan, 30, 1977.



  \bibitem{B-M''}   E. Bierstone and P. Milman. {\em Semianalytic and subanalytic
sets}, Publ. Math. IHES, 67:5--42, 1988.

 \bibitem{BM}   E. Bierstone and P. Milman. {\em Uniformization of analytic spaces}, J. Amer. Math. Soc., 2:801--836, 1989.

 \bibitem{BM1}   E. Bierstone and P. Milman. {\em A simple constructive proof of canonical resolution of singularities}, in T. Mora and C. Traverso, editors, Effective methods in algebraic geometry, pages 11--30. BirkhŠuser, 1991.


\bibitem{BM2} E. Bierstone and D. Milman, {\em  Canonical
desingularization in characteristic zero by blowing up the maximum strata of
a local invariant},  Invent. math. 128, 1997, 207--302.

\bibitem{B-M3} E. Bierstone and D. Milman, {\em  Desingularization algorithms I. Role of exceptional divisors}, Moscow Math. J. 3, 2003, 751--805.


\bibitem{B-M5} E. Bierstone and D. Milman, {\em  Functoriality in resolution of singularities}, Publ. RIMS, Vol. 44, No. 2, 2008, 609--639.

\bibitem{Blanco} R. Blanco, {\em Complexity of Villamayor's algorithm in the non-exceptional monomial
case}, Internat. J. Math., 20(6):659-678, 2009.

\bibitem{BV} A. Bravo and O. Villamayor, {\em A strengthening of resolution of singularities in characteristic zero}. Proc. London Math. Soc. (3) 86 (2003),
 327--357.

\bibitem{BP}  F. Bogomolov and T. Pantev. {\em Weak Hironaka theorem.} Math. Res. Letters, 3(3):299-309, 1996.

\bibitem{Brownawell} D. Brownawell. {\em Bounds for the degrees in the Nullstellensatz.} Ann. Math,
(2) 126, 1987, 577--591.

\bibitem{C}  V. Cossart. {\em Desingularization of embedded excellent surfaces. }T™hoku Math. J., 33:25-33, 1981.


 \bibitem{Cu } S.D. Cutkosky. {\em Resolution of singularities.} Graduate studies in mathematics, ISSN 1065-7339: v. 63


\bibitem{EH} S. Encinas and H. Hauser.  {\em Strong resolution of singularities in characteristic zero. } Comment. Math. Helv. 77 (2002) 821-845
\bibitem{EV}  S. Encinas and O. Villamayor.{\em Good points and constructive resolution of singularities.} Acta Math., 181:109-158, 1998.

\bibitem{EV1}  S. Encinas and O. Villamayor. {\em A course on constructive desingularization and equivariance.} In H. Hauser, J. Lipman, F. Oort, and A. Quir—s, editors, Resolution of Singularities, A research textbook in tribute to Oscar Zariski, volume 181 of Progress in Mathematics. BirkhŠuser, 2000.
\bibitem{EV2}     S. Encinas and O. Villamayor.  {\em A new theorem of desingularization over fields of characteristic zero.} Preprint, 2001.

  \bibitem{G}   J. Giraud. {\em Sur la theorie du contact maximal.} Math. Zeit., 137:285-310, 1974.

\bibitem{Heintz} M.~Giusti, J.~Heintz, J.~Sabia. {\em On the efficiency of effective
Nullstellensaetze.} Comput. Complexity, 3, (1993), 56--95.

\bibitem{Giusti} M.Giusti Some effectivity problems in polynomial ideal theory.
EUROSAM 84 (Cambridge, 1984),  159--171, Lecture Notes in Comput. Sci.,
174, Springer, Berlin, 1984.

\bibitem{Te}    R. Goldin and B. Teissier. {\em Resolving singularities of plane analytic branches with one toric morphism.} Preprint ENS Paris, 1995.

\bibitem{G2} D.~Grigoriev, {\em Computational complexity in polynomial algebra},
Proc. Intern. Congress Math., Berkeley, 1986, vol.~2, 1452--1460.

\bibitem{Gr} A.~Grzegorczyk, {\em Some classes of recursive functions}, Rozprawy Matematiczne,
\textbf{4}, (1953), 1--44.

\bibitem{Hartshorne} R. Hartshorne, {\em  Algebraic Geometry},  Graduate Texts
in Mathematics  52, Springer-Verlag, 1977.


\bibitem{Ha} H. Hauser. {\em The Hironaka theorem on resolution of singularities (or: A proof we always wanted to understand).} Bull Amer. Math. Soc. 40 (2003), 323--403.


\bibitem{Hironaka2} H. Hironaka, {\em  An example of a non-K\"alerian
complex-analytic deformation of K\"ahlerian complex structure},  Annals of
Math. (2), 75, 1962, p. 190-208.

\bibitem{Hir} H. Hironaka, {\em  Resolution of singularities of an
algebraic variety over a field of characteristic zero},  Annals of Math. vol
79, 1964, p. 109-326.
\bibitem{H}    H. Hironaka. {\em Introduction to the theory of infinitely near singular points. }Memo Math. del Inst. Jorge Juan, 28, 1974.

\bibitem{Hir1}    H. Hironaka. {\em Idealistic exponents of singularity.} In Algebraic Geometry. The Johns Hopkins centennial lectures, pages 52-125. Johns Hopkins University Press, Baltimore, 1977.


\bibitem{Jelonek} Z.~Jelonek. {\em On the effective Nullstellensatz.} Invent. Math., 162 (2005),
1--17.

\bibitem{de Jong} A. J. de Jong, {\em  Smoothness, semistability, and
alterations},  Publ. Math. I.H.E.S.  83, 1996, p. 51-93.



\bibitem{Kollar} J.~Koll\'ar. {\em Sharp effective Nullstellensatz.} J. Amer. Math. Soc., 1 (1988),
963--975.
\bibitem{Kollar2} J. Koll\'ar. {\em Lectures on Resolution of singularities},  Princeton University Press, 2007

\bibitem{Kozen} D. Kozen. {\em Efficient resolution of singularities of plane curves},
Lecture Notes in Comput. Sci., {\bf 880} (1994) Springer,  1--11.

\bibitem{L}     J. Lipman.{\em Introduction to the resolution of singularities.} In Arcata 1974, volume 29 of Proc. Symp. Pure Math, 187-229, 1975.

\bibitem{M} K. Matsuki {\em Notes on the inductive algorithm of resolution of singularities}, preprint

\bibitem{Mora} M.~M\"oller, T.~Mora, {\em Upper and lower
bounds for the degree of Groebner
bases},
Lect. Notes Comput. Sci., {\bf 174} (1984), Springer, 172--183.

 \bibitem{O}   T. Oda. {\em Infinitely very near singular points.} Adv. Studies Pure Math., 8:363-404, 1986.

\bibitem{Seidenberg} A.~Seidenberg. {\em On the length of a Hilbert ascending chain.}
Proc. Amer. Math. Soc., 29:443-450, 1971.

\bibitem{Seidenberg74} A.Seidenberg. {\em Constructions in algebra}.  Trans. Amer. Math.
Soc.  197  (1974), 273--313.

\bibitem{Simpson} S.~Simpson. {Ordinal numbers and the Hilbert Basis Theorem.}
J. Symb. Logic, 53: 961-974, 1988.

\bibitem{Teitelbaum} J. Teitelbaum. {\em On the computational complexity of the resolution of plane curve singularities},
Lecture Notes in Comput. Sci., {\bf 358} (1989) Springer,  285--292.

\bibitem{Vladuts} M.~Tsfasman and S.~Vladuts. {\em Algebraic-geometric codes. Mathematics and
its Applications, 58.} Kluwer, 1991.

\bibitem{Vi} O.  Villamayor. {\em Constructiveness of Hironaka's resolution.} Ann. Scient. Ecole Norm. Sup. 4, 22:1-32, 1989.

 \bibitem{Vi1}   O. Villamayor. {\em Patching local uniformizations.} Ann. Scient. Ecole Norm. Sup. 4, 25:629-677, 1992.

 \bibitem{Vi2}   O. Villamayor. {\em Introduction to the algorithm of resolution.} In Algebraic geometry and singularities, La Rabida 1991, pages 123-154. BirkhŠuser, 1996.


\bibitem{Wagner} K.~Wagner and G.~Wechsung. {\em Computational Complexity.} Mathematics
and its Applications, 21, 1986.




\bibitem{Wlod} J. W{\l}odarczyk, {\em Simple Hironaka resolution in characteristic zero},J. Amer. Math. Soc. 18 (2005), 779-822





\end{thebibliography}
\end{document}